\documentclass[leqno]{article}
\usepackage{amsmath, amsthm, amssymb}
\usepackage[T1]{fontenc}

\def\R {\mathbb{R}}
\def\pr{\right )}
\def\le{\left (}
\def\d{\,\mathrm{d}}
\def\f{\varphi}

\usepackage{graphicx}
\usepackage{pst-plot}
\usepackage{pst-math}

\usepackage{authblk}

\author[1]{Tomasz Cie\'{s}lak\thanks{cieslak@impan.pl}}
\author[1]{Jakub Siemianowski\thanks{jsiem@mat.umk.pl}}
\author[2]{Andrzej \'{S}wi\k{e}ch\thanks{swiech@math.gatech.edu}}
\affil[1]{Institute of Mathematics, Polish Academy of Sciences, \'{S}niadeckich 8, Warsaw, Poland,}
\affil[2]{School of Mathematics, Georgia Institute of Technology, Atlanta, GA 30332,USA}

\newtheorem{proposition}{Proposition}[section]
\newtheorem{theorem}[proposition]{Theorem}
\newtheorem{corollary}[proposition]{Corollary}
\newtheorem{lemma}[proposition]{Lemma}
\theoremstyle{definition}
\newtheorem{definition}[proposition]{Definition}
\newtheorem{remark}[proposition]{Remark}

\numberwithin{equation}{section}

\title{Viscosity solutions to an initial value problem for a Hamilton--Jacobi equation with a degenerate Hamiltonian occurring in the dynamics of peakons}

\begin{document}

\maketitle
\begin{abstract}
We consider an initial value problem for a Hamilton--Jacobi equation with a quadratic and degenerate Hamiltonian.
Our Hamiltonian comes from the dynamics of $N$-peakon in the Camassa--Holm equation.
It is given by a quadratic form with a singular positive semi-definite matrix. Such a problem does not fall into the standard theory of viscosity solutions. Also viability related results, sometimes used to deal with degenerate Hamiltonians, do not seem applicable in our case.
We prove the global existence of a viscosity solution by looking at the associated optimal control problem and showing that the value function is a viscosity solution.
The most complicated part is the continuity of a viscosity solution which is obtained in the two-peakon case only.
The source of the difficulties is the non-uniqueness of solutions to the state equation in the optimal control problem.
We prove that the viscosity solution is Lipschitz continuous and unique on some short time interval if the initial condition is Lipschitz continuous.
We end the paper with an example showing the loss of Lipschitz continuity of a viscosity solution in the one-dimensional case.
\end{abstract}
\section{Introduction}\label{Intro}

We consider the initial value problem for the  Hamilton--Jacobi equation of the form
\begin{equation}\label{eq:1.1}
\begin{aligned}
\begin{cases}
u_t(x,t) + \frac{1}{2}E(x)\nabla u(x,t)\cdot \nabla u(x,t) = 0,&x\in \R^N, t\in (0,T)\\
u(x,0) = g(x), &x\in \R^N,
\end{cases}
\end{aligned}
\end{equation}
where $T>0$ is fixed, $p\cdot q$ denotes the inner product between $p$, $q\in \R^N$,
$\nabla u$ denotes the gradient of $u$ with respect to the $x$ variable and $g:\R^N \to \R$ is a given function which is bounded and continuous.
Finally, $E$ is the symmetric matrix
\begin{equation}\label{eq:1.43}
E(x):=
\le
\begin{matrix}
1 & e^{-|x_1-x_2|} & \ldots & e^{-|x_1 - x_N|}\\
e^{-|x_2-x_1|} & 1 & \ldots & e^{-|x_2 - x_N|}\\
\vdots & \vdots & \ddots  & \vdots \\
e^{-|x_N-x_1|} & e ^{-|x_N - x_2|} & \ldots & 1
\end{matrix}
\pr, x= (x_1,x_2, \ldots, x_N) \in \R^N.
\end{equation}
This problem arises in the study of multipeakons. The latter are particular solutions to the Camassa--Holm equation of a form $v(t,x)=\sum_{i=1}^N p_i(t)e^{-|x-q_i(t)|}$,
see \cite{CH}.
Multipeakons play a similar role in the studies of the Camassa--Holm equation to the one played by solitons in KdV equation. The necessary conditions which have to be satisfied by $p_i$ and $q_i$, so that a multipeakon is a solution of the Camassa--Holm equation, is that $p_i$ and $q_i$ solve the following Hamiltonian system, see \cite{K},
\begin{equation}\nonumber
\begin{aligned}
\begin{cases}
\dot{q_i}=\frac{\partial H}{\partial p_i}\\
\dot{p_i}=-\frac{\partial H}{\partial q_i},
\end{cases}
\end{aligned}
\end{equation}
where $p=(p_1,\ldots,p_N)$ and $q=(q_1,\ldots,q_N)$ are vectors, the quadratic Hamiltonian $H$ is given by $H(q,p)=1/2 E(q)p\cdot p$, and
the matrix $E(q)$ has entries $E_{ij}(q)=e^{-|q_i-q_j|}$, $i,j=1,\ldots,N$. Such a symmetric matrix is positive semi-definite, regardless of the dimension, see for instance \cite{CGKM}. Moreover, it is known that the matrix $E$ is positive definite, in particular invertible, if $q_i\neq q_j, i,j=1,\ldots,N, i\not=j$. In such cases an exact form of the inverse $E^{-1}$ is known, see \cite{CGKM}. The dynamics of multipeakons, in particular their collisions, as well as the problem of (unique) continuation past a collision time, has been an area of intensive studies. Let us mention for instance \cite{CH}, where multipeakons have been introduced, \cite{BSS} where the very detailed information concerning occurrence of collisions
in terms of initial data has been given. The methods in \cite{BSS} involve the inverse scattering. Holden and collaborators introduced some other methods to examine the dynamics of multipeakons, see for instance \cite{GH} for the detailed study of a two-peakon case. Finally, let us mention that multipeakons obey the Hamiltonian dynamics (at least as long as the peaks of the multipeakon located at points $q_1(t),\ldots,q_n(t)$ do not collide, so that a Hamiltonian is regular enough). This allows the use of differential geometry methods to study the dynamics of multipeakons, see \cite{CGKM}, \cite{K}. The present paper is a first step in a slightly different direction. The Hamilton--Jacobi equation describes the evolution of the wave propagation front
of the trajectories of a Hamiltonian system. It is our goal to study multipeakons via such an approach. Moreover, Hamilton--Jacobi-like systems occur when dealing with optimal control problems related to multipeakons.
Equation \eqref{eq:1.1} is a Hamilton--Jacobi equation related to the Hamiltonian dynamics of multipeakons, with $E_{ij}=e^{-|q_i-q_j|}$.
The one-dimensional simplifications of \eqref{eq:1.1} are studied in \cite{CW}. However, as noticed in \cite{CW}, the methods used there are peculiar to the one-dimensional setting. The higher-dimensional case requires more advanced approach.

In the present paper we address the question of the existence of viscosity solutions to \eqref{eq:1.1}. In the problem we consider, the matrices $E(x)$ are degenerate whenever $x_i=x_j$. Also, computing the Lagrangian related to the Hamiltonian $H$, we arrive at a singular one, blowing up at the points corresponding to the line $q_i=q_j$. Thus neither classical nor standard  viscosity solution theory, see, e.g., \cite{BCD,CS,CIL,Evans,FS}, is applicable. We have to extend the methods using some tricks, which will lead us to the consideration of the associated optimal control problem having state equation with nonunique solutions. This will require some delicate and careful analysis.

We do not know if equation \eqref{eq:1.1} has a global in time unique viscosity solution. Uniqueness is typically a consequence of the comparison principle
which guarantees that a viscosity subsolution stays below a viscosity supersolution (see \cite{BCD, CIL, FS}). It is not difficult to see that a suitable modification of a standard proof gives comparison principle for equation \eqref{eq:1.1} in any dimension for bounded viscosity subsolutions and supersolutions which are $\alpha$-H\"older continuous in the $x$-variable on every set $\R^n\times(0,T_1), 0<T_1<T$, for some $\alpha>\frac{1}{2}$.
However, it is not expected that viscosity solutions to \eqref{eq:1.1} admit such a high regularity.
Indeed, in the last section we give an explicit formula for a viscosity solution to a one-dimensional simplification of \eqref{eq:1.1}
which is exactly $1/2$-H\"{o}lder continuous in the space variable.
Comparison principle would also work for bounded viscosity subsolutions and supersolutions if the matrices $E(x)$ were positive definite for every $x$.
More sophisticated results about comparison principles for more general equations containing \eqref{eq:1.1} as a model equation can be found in \cite {CDL,DLL1,DLL2}.
Since $\sqrt{E(x)}$ is only $1/2$-H\"older continuous here (see below) and may be degenerate, none of these results applies to our case.
Nevertheless, in the last section we prove local in time uniqueness of viscosity solutions for a slightly more general class of Hamilton--Jacobi problems, see Theorem \ref{thm.uniq:1}.
It turns out that as long as the viscosity solution is Lipschitz continuous it is unique (as mentioned above).
We also show that the viscosity solution of such a more general problem, starting from a Lipschitz initial condition loses Lipschitz continuity and becomes  exactly $1/2$-H\"older continuous at some positive time.

Let us finally mention that our problem is degenerate enough that it does not obey the viability methods, designed to study Hamilton--Jacobi equations exactly when Lagrangians are infinite, see for instance \cite{FPR} or \cite{galbraith} (the latter seems to be the reference covering the most general part of a theory). A straightforward computation shows that an assumption \cite[(A2)]{galbraith} is not satisfied here. Moreover, viability approach does not seem to be extendable to our case without essentially new steps.

In this paper we show in any dimension $N\geq 1$ that the value function of the associated optimal control problem is lower semi-continuous and is a discontinuous viscosity solution of \eqref{eq:1.1} (Theorem \ref{thm:1.2}).
In the two-dimensional case, however, we are able to show that  the value function is continuous despite the degeneracy of $E(x)$.
This can be done thanks to the particular form of $E(x)$.

\subsection{Transformation}\label{subsection:transformation}

The matrix $E(x)$, for $x\in \R^N$, is symmetric and positive semi-definite, so it has the unique square root $\sqrt{E(x)}$ in the class of 
symmetric positive semi-definite matrices.

Let $A$, $B$ be symmetric positive semi-definite $N\times N$ matrices.
By Theorem X.1.1 in \cite{Bhatia} (see also Theorem V.1.9 in \cite{Bhatia}), we have
\begin{equation}\label{eq:1.45}
\|\sqrt{A} - \sqrt{B}\| \leq \sqrt{\|A - B\|},
\end{equation}
where $\|A\|$ denotes the operator norm of the matrix $A$.
By \eqref{eq:1.43}, there is $L_0>0$ such that
\[
\|E(x)-E(y)\|\leq L_0\|x-y\|, \qquad x,\,y\in \R^N.
\]
Combining the above inequalities implies that there is $L_1>0$ such that
\begin{equation}\label{eq:1.44}
\|\sqrt{E(x)} -\sqrt{E(y)}\|\leq L_1\sqrt{\|x-y\|},\qquad x,\,y\in \R^N,
\end{equation}
what shows that the map $x\mapsto \sqrt{E(x)}$ is $1/2$-H\"{o}lder continuous.
In dimension $N=2$, we have the exact formula for $\sqrt{E(x)}$, for $x\in \R^2$,
\begin{eqnarray}\label{eq:1.39}
&&\sqrt{E(x)}:= \\
&&\frac{1}{2}
\le
\begin{matrix}
\sqrt{1 + e^{-|x_1-x_2|}} + \sqrt{1-e^{-|x_1 -x_2|}}  &  \sqrt{1 + e^{-|x_1-x_2|}} - \sqrt{1-e^{-|x_1 -x_2|}}\\
\sqrt{1 + e^{-|x_1-x_2|}} - \sqrt{1-e^{-|x_1 -x_2|}}  &  \sqrt{1 + e^{-|x_1-x_2|}} + \sqrt{1-e^{-|x_1 -x_2|}}
\end{matrix}
\pr ,\nonumber
\end{eqnarray}
which is of particular importance.

Since the matrix $E(x)$ has bounded coefficients (see \eqref{eq:1.43}), there is $C_0>0$ such that $\|E(x)\|\leq C_0$.
We substitute $A = E(x)$ and $B=0$ in \eqref{eq:1.45} to obtain
\begin{equation}\label{eq:1.46}
\|\sqrt{E(x)}\|\leq C_1.
\end{equation}

We have
\[
E(x)v \cdot v = \sqrt{E(x)}v\cdot \sqrt{E(x)}v = \left | \sqrt{E(x)}v\right|^2, \qquad \text{for }x,\,v\in \R^N.
\]
This allows us to reformulate the problem \eqref{eq:1.1} with the use of the Legendre--Fenchel transform. Namely, we have
\begin{equation}\label{eq:1.11}
\frac{1}{2}\left | \sqrt{E(x)} v \right |^2  =\sup_{a\in\R^N} \left \{a \cdot \sqrt{E(x)}v - \frac{1}{2}|a|^2\right \}
=\max_{a\in\R^N} \left \{\sqrt{E(x)}a \cdot v - \frac{1}{2}|a|^2\right\},
\end{equation}
where the last equality is due to the symmetry of $\sqrt{E(x)}$ and the fact that the supremum is attained.
By the above, we rewrite \eqref{eq:1.1} in the form
\begin{equation}\label{eq:1.7}
\begin{cases}
u_t(x,t) + \max_{a\in \R^N}\left\{\sqrt{E(x)}a \cdot \nabla u(x,t) -\frac{1}{2}|a|^2 \right\}= 0,&x\in \R^N,t\in (0,T)\\
u(x,0) = g(x),&x\in \R^N.
\end{cases}
\end{equation}
It is well-known that the terminal value problem equivalent to the above initial value problem is
\[
\begin{aligned}
\begin{cases}
u_t(x,t) - \max_{a\in \R^N} \left\{ \sqrt{E(x)}a \cdot \nabla w (x,t)- \frac{1}{2}|a|^2\right \} = 0,&\text{for }x\in \R^N, \; t\in (0,T),\\
u(x,T) = g(x),&\text{for }x\in\R^N,
\end{cases}
\end{aligned}
\]
or, after a simple modification,
\begin{equation}\label{eq:1.2}
\begin{aligned}
\begin{cases}
u_t(x,t) + \min_{a\in \R^N} \left\{ \sqrt{E(x)}a \cdot \nabla w (x,t)+ \frac{1}{2}|a|^2\right \} = 0,&\text{for }x\in \R^N, \; t\in (0,T),\\
u(x,T) = g(x),&\text{for }x\in\R^N.
\end{cases}
\end{aligned}
\end{equation}
The above equivalence is understood in the sense that $u$ is a viscosity solution of \eqref{eq:1.7} if and only if $w(x,t) := u(x,T-t)$ is a viscosity solution of \eqref{eq:1.2}, see, e.g., Remark (iii) below Thm 2 in \S 10.3 \cite{Evans}.
We use the form \eqref{eq:1.2} or \eqref{eq:1.1} of the Hamilton--Jacobi problem depending on whichever is more convenient for us.

Such a formulation is an initial step of our procedure. We will study the optimal control problem associated to \eqref{eq:1.2} and its value function
and show that it is a viscosity solution of \eqref{eq:1.2}. The problem is that $\sqrt{E(x)}$ is not Lipschitz continuous so the solutions of the state equations of our optimal control problem are not unique
and we do not have uniform continuous dependence estimates with respect to the initial conditions for them. Thus the most difficult part is in proving the continuity of the value function. The full continuity will be achieved only in the dimension $N=2$ in Section \ref{sec:cont} using the special structure of our problem and some
new ad hoc ideas. In dimensions $N\geq 3$ only lower semi-continuity of the value function is shown, see Theorem \ref{thm:1.2}. Once the continuity is established, Section \ref{sec:HJ} will follow standard approach. We will prove the dynamic programming principle and
use it to show that the value function is a viscosity solution of \eqref{eq:1.2}. 
Even though the material of Section \ref{sec:HJ} follows well known arguments which require only minor modifications here, we include full proofs of all results there to make the paper self-contained and easily accessible to readers who are not experts in the dynamic programming approach and the theory of viscosity solutions. In Section \ref{sec:uniq} we will prove a short time uniqueness result for Lipschitz continuous solutions
of the HJB equation \eqref{eq:1.2}.

We recall the definition of viscosity solution of \eqref{eq:1.2}. We refer the readers to \cite{BCD,CS,CIL,Evans,FS} for the basic theory of viscosity
solutions, Hamilton--Jacobi--Bellman equations and their connections to optimal control and calculus of variations problems.

For a function $v:\R^N\times(0,T]\to \R$ we denote by $v^*$ and $v_*$ its upper and lower semi-continuous envelopes, that is
\[
v^*(x,t)=\limsup_{(y,s)\to(x,t)}v(y,s),\quad v_*(x,t)=\liminf_{(y,s)\to(x,t)}v(y,s).
\]
\begin{definition}\label{def:1}
We say that a bounded and upper semi-continuous function $v:\R^N\times(0,T]\to \R$ is a \emph{viscosity subsolution} of the terminal value problem \eqref{eq:1.2} if $v(x,T) \leq g(x)$ for $x\in \R^N$
and for every $\f \in C^1(\R^N\times (0,T))$
\begin{equation}\label{eq:1.37}
\begin{cases}
\text{ if $v-\f$ has a local maximum at $(y_0,t_0)\in \R^N\times (0,T)$, then}\\
\quad\f_t (y_0,t_0) + \min_{a\in \R^N}\left\{ \sqrt{E(y_0)}a\cdot \nabla \f(y_0,t_0) + \frac{1}{2}|a|^2\right\} \geq  0.
\end{cases}
\end{equation}
We say that a bounded and lower semi-continuous function $v:\R^N\times(0,T]\to \R$ is a \emph{viscosity supersolution} of the terminal value problem \eqref{eq:1.2} if $v(x,T) \geq g(x)$ for $x\in \R^N$
and for every $\f \in C^1(\R^N\times (0,T))$
\begin{equation}\label{eq:1.38}
\begin{cases}
\text{ if $v-\f$ has a local minimum at $(y_0,t_0)\in \R^N\times (0,T)$, then}\\
\quad\f_t (y_0,t_0) + \min_{a\in \R^N}\left\{ \sqrt{E(y_0)}a\cdot \nabla \f(y_0,t_0) + \frac{1}{2}|a|^2\right\} \leq  0.
\end{cases}
\end{equation}
A function $v$ is a \emph{viscosity solution} of \eqref{eq:1.2} if it is a viscosity subsolution and a viscosity supersolution of \eqref{eq:1.2}.

\noindent
A function $v$ is a \emph{discontinuous viscosity solution} of \eqref{eq:1.2} if $v^*$ is a viscosity subsolution of \eqref{eq:1.2} and $v_*$ is a viscosity supersolution of \eqref{eq:1.2}.

\end{definition}

\section{Value function}
In this section we introduce the optimal control problem associated to \eqref{eq:1.2} and its value function, which is a candidate for a solution to \eqref{eq:1.2}. We also collect some technical lemmas which are needed in the rest of the paper.

Fix $0\leq t_0\leq T$ and $y_0\in \R^N$.
The value function $v:\R^N\times [0,T]\to \R$ is defined as the following Bolza problem:
\begin{equation}\label{eq:1.4}
v(y_0,t_0) := \inf\left \{\frac{1}{2}\int_{t_0}^T|\alpha(t)|^2\d t + g(x(T)) \right \},
\end{equation}
where the infimum is taken over all measurable controls $\alpha:[0,T]\to \R^N$ and all solutions to the state equation
\[
\tag{$P(\alpha,y_0,t_0)$}
\begin{cases}
\dot{x}(t) = \sqrt{E(x(t))}\alpha(t),&\text{for } t_0 < t < T,\\
x(t_0) = y_0.
\label{eq:1.3}
\end{cases}
\]
Bear in mind that $\sqrt{E(x)}$ is only $1/2$-H\"{o}lder continuous, see \eqref{eq:1.43} and \eqref{eq:1.39}. Hence, in general, solutions $x(t)$ are not unique due to the lack of the Lipschitz condition. Indeed, one easily checks that solutions to ($P(\alpha,y_0,t_0)$) lack uniqueness if $y_0$ belongs to some hyperplane $\{y_i=y_j\}$. Compare with (4) in the remark below.
\begin{remark}\label{rem:1.2}
(1) Note that the infimum in \eqref{eq:1.4} is always finite, since for example $\alpha\equiv 0$ is an admissible control.\\
\indent (2) The infimum in \eqref{eq:1.4} does not change if we restrict ourselves to controls $\alpha\in L^2((0,T),\R^N)$.\\
\indent (3) In view of (2) above, we can only consider controls $\alpha \in  L^2 \le (0,T),\R ^N\pr $.
Hence using the embedding $L^2((0,T),\R^N)\subset L^1((0,T),\R^N)$, for any $y_0$ and $t_0$, there is  an absolutely continuous function $x:[t_0,T]\to \R^N$ which solves \eqref{eq:1.3}, see Theorem XVIII on page 121 of \cite{Walter}.\\
\indent (4) Note that in the case of Lipschitz continuous right-hand side of \eqref{eq:1.3} the above formulation of the optimal control problem becomes the usual one, see \S 10.3.2 of \cite{Evans}.
\end{remark}

We need the following technical lemma. We denote by
$B_{L^2((t,T),\R^N)}(r)$ the closed ball in $L^2((t,T),\R^N)$ centered at $0$ with radius $r>0$.
\begin{lemma}\label{lem:1.5}
Let $y_n\to y_0$ in $\R^N$ and $t_n\to t_0$ in $[0,T]$.
We assume that the sequence of controls $(\alpha_n)$ is bounded, in the sense that $\sup_{n\geq 1}\|\alpha_n\|_{L^2((t_n,T),\R^N)} \leq r < \infty$, and  $x_n$ is any solution of \emph{($P(\alpha_n,y_n,t_n)$)}, $n\geq 1$.
We define $\underline{t} := \inf_{n\geq 0}\{t_n\}$ (note that $\underline{t}\neq t_0$ only if there exist $n$ such that $t_n<t_0$) and
\[
\begin{aligned}
\tilde{\alpha}_n(t)&:=
\begin{cases}
0&\text{for }t\in[\underline{t},t_n),\\
\alpha_n(t)&\text{for }t\in[t_n,T],
\end{cases}\\
\tilde{x}_n(t)&:=
\begin{cases}
y_n&\text{for }t\in[\underline{t},t_n),\\
x_n(t)&\text{for }t\in[t_n,T].
\end{cases}
\end{aligned}
\]
Then there are $\alpha_0$ and $x_0$ such that (up to a subsequence)
\begin{align*}
\tilde{\alpha}_n&\rightharpoonup \alpha_0 \quad\text{ weakly in }L^2((\underline{t},T),\R^N),\\
\tilde{x}_n&\to x_0 \quad \text{ in }C([\underline{t} ,T],\R^N),
\end{align*}
where $\|\alpha_0\|_{L^2((\underline{t},T),\R^N)}\leq r$ and $x_0$ is a solution  of \emph{($P(\alpha_0,y_0,t_0)$)}.
\end{lemma}
\begin{proof}
Certainly, $\tilde{\alpha}_n\in B_{L^2((\underline{t},T),\R^N)}(r)$, and so (up to a subsequence) we may assume that
\begin{equation}\label{eq:1.50}
\tilde{\alpha}_n\rightharpoonup \alpha_0\quad\text{weakly in }L^2((\underline{t},T),\R^N).
\end{equation}
Observe that $\tilde{x}_n$ is a solution of ($P(\tilde{\alpha}_n,y_n,\underline{t})$), i.e.,
\[
\tilde{x}_n (t)= y_n + \int_{\underline{t}}^t \sqrt{E(\tilde{x}_n(s))}\tilde{\alpha}_n(s)\d s,\quad t\in[\underline{t},T],\; n\geq 1.
\]
Using \eqref{eq:1.46}, for every $n\geq 1$,
\[
\sup_{t\in[\underline{t},T]}|x_n(t)|\leq |y_n| + C_1\int_{\underline{t}}^T|\tilde{\alpha}_n(s)|\d s\leq \sup_{n\geq 1} |y_n| +C_1r\sqrt{T-\underline{t}},
\]
so $\{x_n\}$ is bounded in $C([\underline{t},T],\R^N)$.
For any $n\geq 1$ and $\underline{t}\leq t\leq t^\prime\leq T$, we have
\[
|\tilde{x}_n(t^\prime) - \tilde{x}_n(t)| \leq C_1\int_t^{t^\prime} |\tilde{\alpha}_n(s)|\d s\leq C_1r\sqrt{t^\prime - t},
\]
so $\{\tilde{x}_n\}$ is uniformly equicontinuous.
By the Arzela--Ascoli theorem, $\{\tilde{x}_n\}$ is relatively compact in $C([\underline{t},T],\R^N)$.
Passing to a subsequence, we may assume that $\tilde{x}_n\to x_0$ in $C([\underline{t},T],\R^N)$.
Therefore and by \eqref{eq:1.44}, $\sqrt{E(\tilde{x}_n)}\to\sqrt{E(x)}$ in $C([\underline{t},T],\R^{N\times N})$.
Hence and by \eqref{eq:1.50}, for $t\in [\underline{t},T]$, we have
\[
x_0(t)\leftarrow \tilde{x}_n(t) = y_n + \int_{\underline{t}}^t\sqrt{E(\tilde{x}_n(s))}\tilde{\alpha}_n(s)\d s\to y_0 + \int_{\underline{t}}^t\sqrt{E(x_0(s))}\alpha_0(s)\d s.
\]
This shows that $x_0$ is a solution of ($P({\alpha_0,y_0,\underline{t}})$).
On the other hand, since $x_n(t_n) = y_n \to y_0$ and $t_n\to t_0$, we obtain
\[
x_0(t_0) = y_0,
\]
so $x_0$ is also a solution of ($P({\alpha_0,y_0,t_0})$).
\end{proof}
The next lemma guarantees that the set over which we minimize in \eqref{eq:1.4} can be restricted and a minimizer exists.
\begin{lemma}\label{lem:1.3}
For every $(y_0,t_0)\in \R^N\times [0,T]$ there is a control $\alpha_0$ and a solution $x_0$ of \emph{($P(\alpha_0,y_0,t_0)$)} such that $\|\alpha_0\|_{L^2((t_0,T),\R^N)}\leq 2\sqrt{\|g\|_{L^\infty}}$ and
\[
v(y_0,t_0) = \frac{1}{2}\int_{t_0}^{T} |\alpha_0(t)|^2\d t + g(x_0(T)).
\]
In other words, for all $y_0\in\R^2$ and $t_0 \in [0,T]$, the following equality holds:
\[
v(y_0,t_0)= \min\left\{\frac{1}{2}\int_{t_0}^T|\alpha(t)|^2 \d t + g(x(T)) \right\},
\]
where the minimium is taken over all controls $\alpha \in B_{L^2((t_0,T),\R^N)}(2\sqrt{\|g\|_{L^\infty}})$ and all solutions $x$ of \emph{($P(\alpha,y_0,t_0)$)}.
Moreover, the value function $v$ is bounded.
\end{lemma}
\begin{proof}
Fix $y_0\in \R^N$ and $t_0\in[0,T]$. Obviously we have $-\|g\|_{L^\infty}\leq v(y_0,t_0)$ and if
we use the control $\alpha \equiv 0$ in \eqref{eq:1.4}, we get
\begin{equation}\label{eq:1.29}
v(y_0,t_0) \leq g(y_0)\leq \|g\|_{L^\infty},
\end{equation}
so $|v(y_0,t_0)|\leq \|g\|_{L^\infty}$.

For every $n\geq 1$, there is $\alpha_n\in L^2((t_0,T),\R^2)$ and a solution $x_n$  of ($P(\alpha_n,y_0,t_0)$) such that
\begin{equation}\label{eq:1.31}
v(y_0,t_0) + \frac{1}{n}\geq \frac{1}{2}\int_{t_0}^T|\alpha_n(t)|^2 \d t+ g(x_n(T)).
\end{equation}
Therefore, by \eqref{eq:1.29}, we obtain
\begin{equation}\label{eq:1.40}
\frac{1}{2}\int_{t_0}^T|\alpha_n(t)|^2 \d t \leq \frac{1}{n} + 2\|g\|_{L^\infty}, \quad n\geq 1,
\end{equation}
and so, $\alpha_n\in B_{L^2((t_0,T),\R^N)}(r)$, where $r:=\sqrt{2+4\|g\|_{L^\infty}}$, $n\geq 1$.
Taking constant sequences $y_n=y_0$, $t_n=t_0$, $n\geq 1$, and using Lemma \ref{lem:1.5}, there is $\alpha_0\in L^2((t_0,T),\R^N)$ and a solution $x_0$ of ($P(\alpha_0,y_0,t_0)$) such that (up to a subsequence)
\[
\alpha_n\rightharpoonup \alpha_0\;\;\mbox{in}\;\; L^2((t_0,T),\R^N)\;\;\mbox{and}\;\; x_n\to x_0\;\;\mbox{in}\;\;C([t_0,T],\R^N).
\]
By \eqref{eq:1.4}, we obtain
\[
v(y_0,t_0) \leq \frac{1}{2}\int_{t_0}^T|\alpha_0(s)|^2 \d s+ g(x_0(T)).
\]
But \eqref{eq:1.31}, the weak lower semi-continuity of the norm and $x_n(T) \to x_0(T)$ yield
\[
v(y_0,t_0) \geq \liminf_{n\to\infty}\le\frac{1}{2}\int_{t_0}^T|\alpha_n(s)|^2\d s + g(x_n(T))\pr \geq \frac{1}{2}\int_{t_0}^T |\alpha_0(s)|^2 \d s + g(x_0(T)).
\]
The last two inequalities combined give
\begin{equation}\label{eq:1.41}
v(y_0,t_0) = \frac{1}{2}\int_{t_0}^T |\alpha_0(s)|^2 \d s + g(x_0(T)).
\end{equation}
We again use the weak lower semi-continuity of the norm and \eqref{eq:1.40} to get
\[
\|\alpha_0\|_{L^2((t_0,T),\R^N)}\leq \liminf_{n\to\infty}\|\alpha_n\|_{L^2((t_0,T),\R^N)}\leq 2\sqrt{\|g\|_{L^\infty}},
\]
as claimed.
\end{proof}

\section{Continuity of the  value function in dimension $N=2$ and lower semi-continuity in higher dimensions}\label{sec:cont}
In this section we prove one of the main results of the paper, the continuity of the value function. We begin with an auxiliary lemma. Its proof is divided into three steps which are technically different. The reason is that the regularity of the value function depends on whether
$x\in\{x_1=x_2\}$, where $E(x)$ becomes degenerate, or not.
\begin{lemma}\label{lem:1.2}
Fix $y_0\in \R^2$ and $t_0\in[0,T)$.
If $y_n\to y_0$ in $\R^2$, then, for large $n$, there are $t_n\in [t_0,T]$, measurable $\alpha_n:[t_0,t_n)\to\R^2$  and absolutely continuous $x_n:[t_0,t_n]\to \R^2$ satisfying
\begin{enumerate}
\item[\emph{(1)}] $t_n\to t_0^+ $,
\item[\emph{(2)}] $\int_{t_0}^{t_n} |\alpha_n(t)|^2\d t \to 0$,
\item[\emph{(3)}] $x_n$ is a solution of
\[
\begin{cases}
\dot{x}_n(t) = \sqrt{E(x_n(t))}\alpha_n(t),&\text{for }t\in (t_0,t_n),\\
x_n(t_0) = y_n,
\end{cases}
\]
\item[\emph{(4)}] $x_n(t_n) = y_0$.
\end{enumerate}
\end{lemma}
\begin{proof}
Let $y_0=(y_0^1, y_0^2), y_n=(y_n^1, y_n^2), x_0=(x_0^1, x_0^2), x_n=(x_n^1,x_n^2)$. We split the proof into three cases for convenience, depending whether $y_0$ and $y_n$, a sequence approximating $y_0$, are on the line $y^1=y^2$ or not. Notice that the general case can be deduced by combining the three cases together.

\noindent \textbf{Case I:}
\begin{equation}\label{NUMER}
y_0^1=y_0^2\;\;\mbox{and}\;\;y_n^1 \neq y_n^2, n\geq 1.
\end{equation}
We define
\begin{equation}\label{eq:1.17}
t_n := t_0 + \le |y_0^1 - y_n^1| + |y_0^2 - y_n ^2| \pr ^{1/4}.
\end{equation}
Obviously claim (1) holds.
We define $x_n:[t_0,t_n)\to\R^2$ by the formula
\[
x_n(t) = e^{\frac{1}{t_n - t_0} - \frac{1}{t_n - t}} y_n + \le 1 - e^{\frac{1}{t_n-t_0} - \frac{1}{t_n - t}}\pr y_0, \quad t\in [t_0,t_n).
\]
Easy calculations show that $x_n$ is an absolutely continuous (even $C^1$) solution of the equation
\begin{equation}\label{eq:1.19}
\begin{cases}
\dot{x}_n(t) = \frac{1}{(t_n-t)^2}(y_0 -x_n(t)), \quad t\in [t_0,t_n),\\
x_n(t_0) = y_n.
\end{cases}
\end{equation}
Observe that  $x_n(t_n)  = \lim_{t\to t_n ^-}x_n(t) = y_0$, so claim (4) is satisfied.
We have
\begin{equation}\label{eq:1.18}
y_0 - x_n(t) = e^{\frac{1}{t_n - t_0} - \frac{1}{t_n -t}}(y_0 - y_n), \quad t\in[t_0,t_n),
\end{equation}
and
\begin{equation}\label{eq:1.21}
\rho_n(t) := |x_n^1(t) - x_n^2(t)| = e^{\frac{1}{t_n - t_0} - \frac{1}{t_n -t}}|y_n^1 - y_n^2|,\quad t\in [t_0,t_n).
\end{equation}
In particular $x_n^1(t) \neq x_n^2(t)$, for $t\in[t_0,t_n)$, so $\sqrt{E(x_n(t))}^{-1}$ is well-defined and is given by
\[
\sqrt{E(x_n(t))}^{-1} =
\frac{1}{2}\le
\begin{matrix}
\frac{1}{\sqrt{1 + e^{-\rho_n(t)}}} + \frac{1}{\sqrt{1-e^{-\rho_n(t)}}}  &  \frac{1}{\sqrt{1 + e^{-\rho_n(t)}}} - \frac{1}{\sqrt{1-e^{-\rho_n(t)}}}\\
\frac{1}{\sqrt{1 + e^{-\rho_n(t)}}} - \frac{1}{\sqrt{1-e^{-\rho_n(t)}}} & \frac{1}{\sqrt{1 + e^{-\rho_n(t)}}} + \frac{1}{\sqrt{1-e^{-\rho_n(t)}}}
\end{matrix}
\pr,
\]
for $t\in [t_0,t_n).$

Hence, we may define $\alpha_n:[t_0,t_n)\to \R^2$ by
\begin{equation}\label{eq:1.23}
\alpha_n(t) := \frac{1}{(t_n-t)^2}\sqrt{E(x_n(t))}^{-1}(y_0-x_n(t)),\quad t\in [t_0,t_n).
\end{equation}
Observe that $\sqrt{E(x_n(t))} \alpha_n(t) = \frac{1}{(t_n-t)^2}(y_0-x_n(t))$, i.e., the right-hand side of \eqref{eq:1.19}.
Hence claim (3) is satisfied.

All that remains is to show claim (2).
Using \eqref{eq:1.18} and \eqref{NUMER}, after some computations, we obtain
\[
\alpha_n(t) = \frac{e^{\frac{1}{t_n - t_0} - \frac{1}{t_n - t}}}{2(t_n-t)^2}
\le
\begin{matrix}
\vspace{0.3cm}
\dfrac{y_0^1 - y_n^1 + y_0^2 - y_n^2}{\sqrt{1 + e^{-\rho_n(t)}}} - \dfrac{y_n^1 - y_n^2}{\sqrt{1 - e^{-\rho_n(t)}}}\\
\dfrac{y_0^1 - y_n^1 + y_0^2 - y_n^2}{\sqrt{1 + e^{-\rho_n(t)}}} + \dfrac{y_n^1 - y_n^2}{\sqrt{1 - e^{-\rho_n(t)}}}
\end{matrix}
\pr, \quad t\in [t_0,t_n).
\]
We have
\begin{equation}\label{eq:1.22}
|\alpha_n(t)|^2 = \frac{1}{2}\underbrace{\frac{e^{\frac{2}{t_n - t_0} - \frac{2}{t_n - t}}(y_0^1-y_n^1+y_0^2-y_n^2)^2}{(t_n-t)^4\le 1 + e^{-\rho_n(t)} \pr}}_{=: I_{1,n}(t)} + \frac{1}{2}\underbrace{\frac{e^{\frac{2}{t_n - t_0} - \frac{2}{t_n - t}}(y_n^1- y_n^2)^2}{(t_n-t)^4\le 1- e^{-\rho_n(t)}\pr }}_{=:I_{2,n}(t)}, \quad t\in [t_0,t_n).
\end{equation}
We estimate each of the above terms separately.
Note that $\rho_n(t) \geq 0$ so
\[
I_{1,n}(t)\leq \frac{e^{\frac{2}{t_n - t_0} - \frac{2}{t_n - t}}(y_0^1-y_n^1+y_0^2-y_n^2)^2}{(t_n-t)^4} = : h_n(t).
\]
Next, we compute
\[
h_n'(t)=e^{\frac{2}{t_n-t_0}}(y_0^1-y_n^1+y_0^2-y_n^2)^2\frac{-2e^{\frac{-2}{t_n-t}}(t_n-t)^2[1-2(t_n-t)]}{(t_n-t)^8}.
\]
Hence, the function $h_n$ has a negative derivative on $[t_0,t_n)$, provided that $n$ is sufficiently large.
Thus, $h_n$ attains maximum on $[t_0,t_n)$ at $t_0$, and so we gain
\begin{equation}\label{eq:1.20}
\sup_{t\in[t_0,t_n)}I_{1,n}(t) \leq h_n(t_0)= \frac{(y_0^1-y_n^1+y_0^2-y_n^2)^2}{(t_n-t_0)^4} \to 0, \text{ as }n\to \infty,
\end{equation}
the last convergence holds due to \eqref{eq:1.17}.

Next, note that
\[
I_{2,n}(t) = \frac{1}{2}\frac{\rho_n^2(t)}{(t_n-t)^4\le 1-e^{-\rho_n(t)}\pr }\;.
\]
Due to \eqref{NUMER}, $\rho_n(t)>0$ and using the inequality
\[
\frac{1}{1-e^{-x}}\leq \frac{1+x}{x},\quad  x > 0,
\]
we get
\[
I_{2,n}(t) \leq \frac{1}{2}\frac{\rho_n^2(t)(1+\rho_n(t))}{(t_n-t)^4\rho_n(t)} = \frac{1}{2}\frac{\rho_n(t)(1+\rho_n(t))}{(t_n-t)^4},\quad t\in [t_0,t_n).
\]
In view of \eqref{NUMER} and since $y_n\to y_0$, $\sup_{t\in[t_0,t_n)}\rho_n(t) \to 0$ as $n\to \infty$. Hence, for sufficiently large $n$, we obtain
\[
I_{2,n}(t) \leq \frac{\rho_n(t)}{(t_n-t)^4}, \quad t\in[t_0,t_n).
\]
By \eqref{eq:1.21} and the above inequality
\[
\int_{t_0}^{t_n} I_{2,n}(t)\d t \leq |y_n^1-y_n^2|e^\frac{1}{t_n-t_0}\int_{t_0}^{t_n}\frac{e^{-\frac{1}{t_n-t}}}{(t_n-t)^4}\d t.
\]
We change the variables $e^{-\frac{1}{t_n-t}} = s $ and get
\[
\begin{aligned}
\int_{t_0}^{t_n} I_{2,n}(t)\d t &\leq |y_n^1-y_n^2|e^\frac{1}{t_n-t_0}\int_0^{e^{-\frac{1}{t_n-t_0}}}\ln^2(s)\d s\\
&= |y_n^1-y_n^2|e^\frac{1}{t_n-t_0}\Big [s(\ln^2(s)-2\ln(s) +2) \Big ]_{s= 0}^{s=e^{-\frac{1}{t_n-t_0}}}\\
&= |y_n^1-y_n^2|e^\frac{1}{t_n-t_0}\left [e^{-\frac{1}{t_n-t_0}}\le \frac{1}{(t_n -t_0)^2} + 2\frac{1}{t_n-t_0} + 2  \pr \right ]\\
&= \frac{|y_n^1-y_n^2|}{(t_n - t_0)^2}+2\frac{|y_n^1-y_n^2|}{t_n- t_0} + 2|y_n^1 - y_n^2|.
\end{aligned}
\]
But $|y_n^1 - y_n^2|\leq |y_n^1- y_0^1| + |y_0^2 - y_n^2|$ due to \eqref{NUMER}. Hence, we obtain
\[
\int_{t_0}^{t_n} I_{2,n}(t)\d t  \leq \frac{|y_n^1- y_0^1|+|y_0^2 - y_n^2|}{(t_n - t_0)^2}+2\frac{|y_n^1- y_0^1|+|y_0^2 - y_n^2|}{t_n- t_0} + 2|y_n^1 - y_n^2| \stackrel{n\to \infty}\longrightarrow 0,
\]
by \eqref{eq:1.17}.
This, together with \eqref{eq:1.20} and \eqref{eq:1.22}, proves claim (2).

\vspace{0.5cm}

\noindent \textbf{Case II: $y_0^1 = y_0^2$ and $y_n^1 = y_n^2$, $n\geq 1$.}

Let $t_n := t_0 + \sqrt{2}|y_0^1 - y_n^1|$.
Obviously, (1) is satisfied.
Let $x_n :[t_0,t_n] \to \R^2$ be the solution of
\begin{equation}\label{eq:1.24}
\begin{cases}
\dot{x}_n(t) = \frac{\mathrm{sgn}(y_0^1-y_n^1)}{\sqrt{2}}
\le
\begin{matrix}
1\\
1
\end{matrix}
\pr, &t\in [t_0,t_n],\\
x_n(t_0) = y_n.
\end{cases}
\end{equation}
Then, for $t\in[t_0,t_n]$,
\[
x_n(t) = y_n + (t-t_0)\frac{\mathrm{sgn}(y_0^1-y_n^1)}{\sqrt{2}}\le
\begin{matrix}
1\\
1
\end{matrix}
\pr,
\]
so in particular
\begin{equation}\label{eq:1.25}
x_n^1(t) = x_n^2(t),\quad t\in [t_0,t_n].
\end{equation}
Moreover, $x_n(t_n) = y_0$, so (4) is satisfied.

We define $\alpha_n:[t_0,t_n]\to \R^2$ by
\[
\alpha_n(t) = \mathrm{sgn}(y_0^1-y_n^1)
\le
\begin{matrix}
1\\
0
\end{matrix}
\pr, \quad t\in [t_0,t_n],
\]
so (2) follows immediately.

By \eqref{eq:1.25}, we have
\[
\sqrt{E(x_n(t))} = \frac{1}{\sqrt{2}}
\le
\begin{matrix}
1 & 1\\
1 & 1
\end{matrix}
\pr,\quad t\in [t_0,t_n].
\]
Thus, for $t\in[t_0,t_n]$,
\[
\sqrt{E(x_n(t))}\alpha_n(t)  = \frac{\mathrm{sgn}(y_0^1-y_n^1)}{\sqrt{2}}
\le
\begin{matrix}
1\\
1
\end{matrix}
\pr = \dot{x}_n(t),
\]
so (3) is satisfied.

\vspace{0.5cm}

\noindent \textbf{Case III: $y_0^1 \neq y_0^2$.}

We define $t_n := t_0 + |y_0 - y_n|$, so (1) holds.
Without loss of generality we assume that $y_n\neq y_0$, $n\geq 1$.
Let $x_n:[t_0,t_n]\to \R^2$ be the solution of
\begin{equation}\label{eq:1.26}
\begin{cases}
\dot{x}_n(t) = \frac{1}{|y_0-y_n|}(y_0-y_n), &\text{for }t\in [t_0,t_n],\\
x_n(t_0) = y_n.
\end{cases}
\end{equation}
Then, for $t\in[t_0,t_n]$,
\[
x_n(t) = y_n + \frac{t-t_0}{|y_0-y_n|}(y_0 - y_n)
\]
and we see that (4) is satisfied.
We define $\alpha_n:[t_0,t_n]\to \R^2$ by
\[
\alpha_n(t) := \frac{1}{|y_0-y_n|}\sqrt{E(x_n(t))}^{-1}(y_0-y_n), \quad t\in [t_0,t_n].
\]
For sufficiently large $n$, $x_n([t_0,t_n])\subset B(y_0,\delta)$, where $\delta>0$ is chosen so small that
\[
\overline{B(y_0,\delta)}\cap \{(x_1,x_2)\in \R^2 \mid x_1 = x_2\} = \emptyset.
\]
As a result, the matrix $\sqrt{E(x_n(t))}^{-1}$ is well-defined and its coefficients are bounded by some constant uniformly with respect to $n$ and $t\in[t_0,t_n]$.
Thus, we have
\[
|\alpha_n(t)| \leq c,\quad t\in [t_0,t_n],
\]
for a suitable constant $c>0$.
This immediately implies (2).
By the very definition of $\alpha_n$ and by \eqref{eq:1.26} we obtain (3).
\end{proof}
We now have the necessary technical tools to prove the continuity of the value function.
We begin with the upper semi-continuity.
\begin{theorem}\label{thm:1.1}
The value function $v:\R^2\times[0,T]\to\R$ is upper semi-continuous.
\end{theorem}
\begin{proof}
First, we show that $v$ is upper semi-continuous with respect to the $y$-variable.
Fix $y_0\in \R^2$ and $t_0 \in[0,T)$ (the case $t_0 = T$ is trivial).
We choose $y_n\to y_0$ such that
\[
\lim_{n\to \infty} v(y_n,t_0) = \limsup_{y\to y_0}v(y,t_0).
\]
We aim to prove the following inequality
\begin{equation}\label{eq:1.16}
\lim_{n\to \infty} v(y_n,t_0) \leq v(y_0,t_0).
\end{equation}
By Lemma \ref{lem:1.3}, there is $\alpha_0\in L^2((0,T),\R^2)$ and a solution $x_0$ of ($P(\alpha_0,y_0,t_0)$) such that
\begin{equation}\label{eq:1.27}
v(y_0,t_0) = \frac{1}{2}\int_{t_0}^T |\alpha_0(t)|^2 \d t + g(x_0(T))).
\end{equation}
Since $y_n\to y_0$, we apply Lemma \ref{lem:1.2} and obtain $(t_n)$, $(\alpha_n)$ and $(x_n)$ such that
\[
\begin{aligned}
&t_n\to t_0^+, \quad \int_{t_0}^{t_n} |\alpha_n(t)|^2\d t \to 0,\quad n \to \infty,\\
&\begin{cases}
\dot{x}_n(t) = \sqrt{E(x_n)}\alpha_n(t),&\text{for }t\in (t_0,t_n),\\
x_n(t_0) = y_n,
\\
x_n(t_n) = y_0.
\end{cases}
\end{aligned}
\]
Now, we define
\[
\tilde{x}_n(t) :=
\begin{cases}
x_n(t),&\text{for }t\in[t_0,t_n),\\
x_0(t -t_n +t_0), &\text{for }t\in[t_n,T]
\end{cases}
\]
and
\[
\tilde{\alpha}_n(t):=
\begin{cases}
\alpha_n(t), &\text{for }t\in[t_0,t_n),\\
\alpha_0(t - t_n + t_0), &\text{for }t\in[t_n,T].
\end{cases}
\]
Clearly, $\tilde{x}_n $ is absolutely continuous and is a solution of ($P(\tilde{\alpha}_n,y_n,t_0)$), hence
\begin{equation}\label{eq:1.28}
v(y_n,t_0) \leq \frac{1}{2}\int_{t_0}^T|\tilde{\alpha}_n(t)|^2 \d t - g(\tilde{x}_n(T)).
\end{equation}
Using the properties of $t_n$, $\tilde{\alpha}_n$ and $\tilde{x}_n$, we obtain
\[
\begin{aligned}
\int_{t_0}^T|\tilde{\alpha}_n(t)|^2 \d t &= \int_{t_0}^{t_n}|\alpha_n(t)|^2 \d t  + \int_{t_n}^T|\alpha_0(t-t_n+t_0)|^2 \d t \to \int_{t_0}^T|\alpha_0(t)|^2 \d t ,\\
g(\tilde{x}_n(T)) &= g(x_0(T-t_n+t_0))\to g(x_0(T)).
\end{aligned}
\]
Thus, \eqref{eq:1.28}, the above convergences and \eqref{eq:1.27}, yield
\[
\lim_{n\to\infty} v(y_n,t_0)  \leq v(y_0,t_0).
\]

We now show that $v$ is upper semi-continuous with respect to both variables.
Fix $(y_0,t_0)$ and $(y_n,t_n)\in\R^2\times [0,T]$ such that $(y_n,t_n)\to (y_0,t_0)$ and
\[
\lim_{n\to\infty} v(y_n,t_n) = \limsup_{(y,t)\to (y_0,t_0)}v(y,t).
\]
We will show that
\[
\lim_{n\to\infty} v(y_n, t_n)  \leq v(y_0,t_0).
\]
Note that
\begin{equation}\label{eq:1.34}
v(y_n,t_n) - v(y_0,t_0) = \underbrace{v(y_n,t_n) - v(y_n,t_0)}_{=:I_n} + \underbrace{v(y_n,t_0) - v(y_0,t_0)}_{=:J_n}
\end{equation}
By the first part of the proof $v$ is upper semicontinuous with respect to $y$, so $\limsup_{n\to\infty} J_n \leq 0$.
Thus, if we show that $\limsup_{n\to\infty}I_n \leq 0$, we are done.
To deal with $I_n$ we consider two cases.

\noindent \textbf{Case I: $t_n\to t_0^-$.}
For every $n\geq 1$, by Lemma \ref{lem:1.3}, there are a control $\alpha_n$ and a solution $x_n$ of ($P(\alpha_n, y_n, t_0)$) such that
\begin{equation}\label{wazne}
v(y_n,t_0) = \frac{1}{2}\int_{t_0}^T|\alpha_n(s)|^2\d s + g(x_n(T)).
\end{equation}
We define
\[
\begin{aligned}
\tilde{\alpha}_n(t)&:=
\begin{cases}
0&\text{for }t\in[t_n,t_0),\\
\alpha_n(t)&\text{for }t\in[t_0,T],
\end{cases}\\
\tilde{x}_n(t)&:=
\begin{cases}
y_n&\text{for }t\in[t_n,t_0),\\
x_n(t)&\text{for }t\in[t_0,T].
\end{cases}
\end{aligned}
\]
Obviously, $\tilde{x}_n$ is a solution of ($P(\tilde{\alpha}_n, y_n, t_n)$), so by \eqref{eq:1.4},
\[
v(y_n,t_n) \leq \frac{1}{2}\int_{t_n}^T|\tilde{\alpha}_n(s)|^2\d s + g(\tilde{x}_n(T)).
\]
Consequently, we have
\[
v(y_n,t_n) \leq \frac{1}{2}\int_{t_0}^T|\alpha_n(s)|^2 \d s + g(x_n(T)) = v(y_n,t_0).
\]
where the last identity follows by \eqref{wazne}. Hence, $I_n\leq 0$, for $n\geq 1$.

\noindent \textbf{Case II: $t_n\to t_0^+$.}
For every $n\geq 1$, by Lemma \ref{lem:1.3}, there exist $\alpha_n\in B_{L^2((t_0,T),\R^2)}(2\sqrt{\|g\|_{L^\infty}})$ and a solution $x_n$ of ($P(\alpha_n, y_n, t_0)$) such that
\begin{equation}\label{eq:1.36}
v(y_n,t_0) = \frac{1}{2}\int_{t_0}^T|\alpha_n(s)|^2\d s + g(x_n(T)).
\end{equation}
We define, for $t\in [t_n,T]$,
\[
\begin{aligned}
\tilde{\alpha}_n(t):= \alpha_n(t-t_n+t_0)\\
\tilde{x}_n(t):= x_n(t-t_n+t_0).
\end{aligned}
\]
Then $\tilde{x}_n$ is a solution of ($P(\tilde{\alpha}_n,y_n,t_n)$), so
\begin{equation}\label{eq:1.35}
\begin{aligned}
v(y_n, t_n) &\leq \frac{1}{2}\int_{t_n}^T|\tilde{\alpha}_n(s)|^2 \d s + g(\tilde{x}_n(T)) \\
&= \frac{1}{2}\int_{t_0}^{T-t_n+t_0}|\alpha_n(s)|^2 \d s + g(x_n(T - t_n + t_0))\\
&\leq \frac{1}{2}\int_{t_0}^{T}|\alpha_n(s)|^2 \d s + g(x_n(T - t_n + t_0)).
\end{aligned}
\end{equation}
Let us choose a subsequence $(y_{n_k},t_{n_k})$ such that
\[
\lim_{k\to\infty}(v(y_{n_k},t_{n_k}) - v(y_{n_k},t_0) )= \limsup_{n\to\infty}(v(y_n,t_n) - v(y_n,t_0)).
\]
Since $y_{n_k}\to y_0$, $\alpha_{n_k}\in B_{L^2((t_0,T),\R^2)}(2\sqrt{\|g\|_{L^\infty}})$ and $x_{n_k}$ is a solution of ($P(\alpha_{n_k},y_{n_k},t_0)$), by Lemma \ref{lem:1.5}, $\{x_{n_k}\}$ is relatively compact in $C([t_0,T],\R^2)$.
After passing to a subsequence (still denoted by $x_{n_k}$) we may assume that $x_{n_k}\to x_0$ in $C([t_0,T],\R^2)$.
Then, by \eqref{eq:1.36} and \eqref{eq:1.35}, for every $k\geq 1$,
\[
v(y_{n_k},t_{n_k}) - v(y_{n_k},t_0)\leq g(x_{n_k}(T - t_{n_k} + t_0)) - g(x_{n_k}(T)),
\]
so
\[
\limsup_{n\to\infty} I_n = \lim_{k\to\infty} \le v(y_{n_k},t_{n_k}) - v(y_{n_k},t_0)\pr \leq 0. \qedhere
\]
\end{proof}

\begin{theorem}\label{thm:1.2}
The value function $v:\R^N\times [0,T]\to \R$ is lower semi-continuous.
\end{theorem}
\begin{remark}
Note that the lower semi-continuity of $v$ is shown in any dimension $N$.
\end{remark}
\begin{proof}
Fix $y_0\in \R^N$ and $t_0\in [0,T]$.
Choose $y_n\to y_0$ and $t_n \to t_0$ so that
\[
\lim_{n\to\infty}v(y_n,t_n) = \liminf_{(y,t)\to(y_0,t_0)}v(y,t).
\]
We will show that
\[
\lim_{n\to\infty} v(y_n,t_n) \geq v(y_0,t_0).
\]
For every $n\geq 1$, by Lemma \ref{lem:1.3}, there are $\alpha_n \in B_{L^2((t_n,T),\R^N)}(2\sqrt{\|g\|_{L^\infty}})$ and a solution $x_n$ of ($P(\alpha_n,y_n,t_n)$) such that
\[
v(y_n,t_n) = \frac{1}{2}\int_{t_n}^T|\alpha_n(s)|^2\d s + g(x_n(T)).
\]
We use Lemma \ref{lem:1.5} and its notation.
We may assume, after passing to a subsequence, that
\[
\begin{aligned}
\tilde{\alpha}_{n_k}&\rightharpoonup \alpha_0 \quad\text{in }L^2((\underline{t},T),\R^N),\\
\tilde{x}_{n_k}&\to x_0 \quad\text{in }C([\underline{t},T],\R^N),
\end{aligned}
\]
and $x_0$ is a solution of ($P(\alpha_0,y_0,t_0)$).
By the above and using the properties of $\tilde{\alpha}_{n_k}$, $\tilde{x}_{n_k}$ and $\underline{t}$, we obtain
\[
\begin{aligned}
\lim_{n\to\infty}v(y_n,t_n) &= \lim_{k\to\infty}v(y_{n_k},t_{n_k}) = \liminf_{k\to\infty}v(y_{n_k},t_{n_k})\\
&=\liminf_{k\to\infty}\left[\frac{1}{2}\int_{t_{n_k}}^T|\alpha_{n_k}(s)|^2\d s + g(x_{n_k}(T))\right]\\
&=\liminf_{k\to\infty}\left[\frac{1}{2}\int_{\underline{t}}^T|\tilde{\alpha}_{n_k}(s)|^2\d s + g(\tilde{x}_{n_k}(T))\right]\\
&\geq \frac{1}{2}\int_{\underline{t}}^T|\alpha_0(s)|^2\d s + g(x_0(T))\\
&\geq \frac{1}{2}\int_{t_0}^T|\alpha_0(s)|^2\d s + g(x_0(T))\\
&\geq v(y_0,t_0),
\end{aligned}
\]
where we used the weak lower semi-continuity of the norm and \eqref{eq:1.4}.
\end{proof}

We conclude this section with a theorem summarizing the obtained results.
\begin{theorem}\label{thm:1.3}
If $g$ is continuous and bounded, then the value function $v:\R^2\times[0,T]\to \R$ defined by \eqref{eq:1.4} is continuous and
$\|v\|_{L^\infty}\leq \|g\|_{L^\infty}$.
\end{theorem}
\begin{proof}
See Theorems \ref{thm:1.1}, \ref{thm:1.2} and Lemma \ref{lem:1.3}.
\end{proof}

\section{The PDE for the value function}\label{sec:HJ}
This section is rather standard and follows well known arguments which require only minor modifications. However we include the proofs to make the paper self-contained and make it easily readable by people who are not experts in the dynamic programming approach and the theory of viscosity solutions. We first show the dynamic programming principle which is the key step in the proof
that the value function is a viscosity solution to the Hamilton--Jacobi equation. The proof of the dynamic programming principle below follows the proof of Theorem 1, \S 10.3 of \cite{Evans}.
\begin{lemma}[Dynamic Programming Principle]\label{lem:1.1}
For every $h>0$ such that $t_0 + h <T$, we have
\[
v(y_0,t_0) = \inf\left\{ \int_{t_0}^{t_0+h}\frac{1}{2}|\alpha(t)|^2\d t + v(x(t_0+h),t_0+h)\right\},
\]
where the infimum is taken over all measurable controls $\alpha$ and all solutions $x$ of \eqref{eq:1.3}.
\end{lemma}
\begin{remark}\label{rem:1.1}
The infimum above does not change if we restrict ourselves to square integrable controls $\alpha$, cf. Remark \ref{rem:1.2}.
\end{remark}
\begin{proof}
Fix any square integrable control $\alpha_1$ and any solution $x_1$ of ($P(\alpha_1,y_0,t_0)$).
By Lemma \ref{lem:1.3}, there is some control $\alpha_2$ and a solution $x_2$ of ($P(\alpha_2,x_1(t_0 +h),t_0+h)$) such that
\begin{equation}\label{eq:1.5}
v(x_1(t_0+h),t_0+h)  = \int_{t_0+h}^T\frac{1}{2}|\alpha_2(t)|^2\d t + g(x_2(T)).
\end{equation}
We define the control $\alpha_3:[t_0,T]\to\R^N$ and the function $x_3:[t_0,T]\to \R^N$ as follows
\[
\begin{aligned}
\alpha_3(t) &:=
\begin{cases}
\alpha_1(t)&\text{for }t_0\leq t<t_0+h,\\
\alpha_2(t)&\text{for } t_0+h \leq t \leq T,
\end{cases}\\
x_3(t) &:=
\begin{cases}
x_1(t)&\text{for }t_0\leq t<t_0+h,\\
x_2(t)&\text{for } t_0+h \leq t \leq T.
\end{cases}
\end{aligned}
\]
It is straightforward to see that $x_3$ is absolutely continuous and solves ($P(\alpha_3,t_0,y_0)$).
By \eqref{eq:1.4}, we have
\[
\begin{aligned}
v(y_0,t_0)&\leq \int_{t_0}^T\frac{1}{2}|\alpha_3(t)|^2 \d t + g(x_3(T))\\
&=\int_{t_0}^{t_0+h}\frac{1}{2}|\alpha_1(t)|^2 \d t +\int_{t_0+h}^T\frac{1}{2}|\alpha_2(t)|^2\d t + g(x_2(T))\\
&= \int_{t_0}^{t_0+h}\frac{1}{2}|\alpha_1(t)|^2 \d t + v(x_1(t_0+h),t_0+h),
\end{aligned}
\]
where the last equality follows from \eqref{eq:1.5}.
Since the control $\alpha_1$ and the solution $x_1$ of ($P(\alpha_1,y_0,t_0)$) was arbitrary, we obtain
\[
v(y_0,t_0) \leq \inf\left\{ \int_{t_0}^{t_0+h}\frac{1}{2}|\alpha(t)|^2\d t + v(x(t_0+h),t_0+h) \right\},
\]
where the infimum is taken over all measurable controls $\alpha$ and all solutions $x$ of
($P(\alpha,y_0,t_0)$) (see Remark \ref{rem:1.1}).

We now prove the opposite inequality.
By Lemma \ref{lem:1.3}, there is a control $\alpha_4$ and a solution $x_4$ of ($P(\alpha_4,y_0,t_0)$) such that
\[
v(t_0, y_0)= \int_{t_0}^T\frac{1}{2}|\alpha_4(t)|^2\d t + g(x_4(T)).
\]
Note that the restriction $ x_4\vert_{[t_0+h,T]}$ is a solution of ($P(\alpha_4,x_4(t_0+h),t_0+h))$), so \eqref{eq:1.4} yields
\[
v(x_4(t_0+h),t_0+h) \leq \int_{t_0+h}^T\frac{1}{2}|\alpha_4(t)|^2\d t + g(x_4(T)).
\]
Gathering the above, we get
\[
\begin{aligned}
v(y_0,t_0) &= \int_{t_0}^T\frac{1}{2}|\alpha_4(t)|^2\d t + g(x_4(T))\\
&=\int_{t_0}^{t_0+h}\frac{1}{2}|\alpha_4(t)|^2\d t  + \int_{t_0+h}^T\frac{1}{2}|\alpha_4(t)|^2\d t + g(x_4(T))\\
&\geq \int_{t_0}^{t_0+h}\frac{1}{2}|\alpha_4(t)|^2\d t+ v(x_4(t_0+h),t_0+h).
\end{aligned}
\]
Hence, taking the infimum over all controls $\alpha$ and all solutions $x$ of ($P(\alpha,y_0,t_0)$) yields
\[
v(y_0,t_0) \geq \inf\left\{\int_{t_0}^{t_0+h}\frac{1}{2}|\alpha(t)|^2\d t+ v(x(t_0+h),t_0+h)\right\}.\qedhere
\]
\end{proof}

\begin{corollary}\label{cor:1.1}
For every $h>0$ such that $t_0 + h <T$, we have
\[
v(y_0,t_0) = \min\left\{ \int_{t_0}^{t_0+h}\frac{1}{2}|\alpha(t)|^2\d t + v(x(t_0+h),t_0+h)\right\},
\]
where the minimum is taken over all controls $\alpha$ with $\|\alpha\|_{L^2((t_0,T),\R^N)}\leq 2\sqrt{\|g\|_{L^\infty}}$ and all solutions $x$ of \emph{($P(\alpha,y_0,t_0)$)}.
\end{corollary}
\begin{proof}
By Lemma \ref{lem:1.3}, there is $\alpha_0\in B_{L^2((t_0,T),\R^N)}(2\sqrt{\left\|g\right\|_{L^\infty}})$ and
a solution $x_0$ of ($P(\alpha_0,y_0,t_0)$) such that
\begin{equation}\label{dpp1}
v(y_0,t_0) =\frac{1}{2}\int_{t_0}^T|\alpha_0(s)|^2 \d s + g(x_0(T)).
\end{equation}
Since $x_0$ is some solution of ($P(\alpha_0,y_0,t_0)$), Lemma \ref{lem:1.1} implies that
\begin{equation}\label{dpp2}
v(y_0,t_0)\leq \frac{1}{2}\int_{t_0}^{t_0+h}|\alpha_0(s)|^2\d s+v(x_0(t_0+h),t_0+h).
\end{equation}
Combining \eqref{dpp1} and \eqref{dpp2}, we obtain
\[
\frac{1}{2}\int_{t_0+h}^T|\alpha_0(s)|^2\d s + g(x_0(T))\leq v(x_0(t_0+h),t_0+h).
\]
On the other hand, $x_0|_{[t_0+h,T]}$ is a solution of ($P(\alpha_0,x_0(t_0+h),t_0+h)$), and so, by the definition,
\[
v(x_0(t_0+h),t_0+h)\leq \frac{1}{2}\int_{t_0+h}^T|\alpha_0(s)|^2\d s + g(x_0(T)),
\]
hence
\[
v(x_0(t_0+h),t_0+h)= \frac{1}{2}\int_{t_0+h}^T|\alpha_0(s)|^2\d s + g(x_0(T))
\]
and, in view of \eqref{dpp1}, the proof of the corollary is complete.
\end{proof}

We can now prove that the value function is a viscosity solution to the Hamilton--Jacobi equation, see Definition \ref{def:1}. The proof follows standard arguments
(see for instance \cite{BCD,Evans,FS})
with some adjustments to comply with our case.
\begin{theorem}
Let $v:\R^N\times[0,T]\rightarrow \R$ be a value function defined in \eqref{eq:1.4}. Then, in dimension $N=2$, $v$ is a viscosity solution to \eqref{eq:1.2}. In dimensions $N\geq 3$, $v$ is a lower semi-continuous discontinuous viscosity solution to \eqref{eq:1.2}.
\end{theorem}
\begin{proof}
The value function $v$ is lower semi-continuous and bounded in every dimension $N\geq 2$, according to Theorem \ref{thm:1.2} and Lemma \ref{lem:1.3}.
In dimension $N=2$, the value function $v$ is additionally continuous by Theorem \ref{thm:1.1}.

By the very definition \eqref{eq:1.4}, we get $v(y,T)= g(y)$, $y\in \R^N$.

Let $\f\in C^1(\R^N\times(0,T))$ be such that
\[
\text{$v^*-\f$ has a local maximum at $(y_0,t_0)\in \R^N\times (0,T)$,}
\]
that is
\begin{equation}\label{ass:1.1}
v^*(y,t)-\f (y,t)\leq v^*(y_0,t_0)-\f(y_0,t_0) \quad \mbox{for}\,\,|t-t_0|<\delta,|y-y_0|<\delta,
\end{equation}
for some $\delta>0$.
We claim that
\begin{equation}\label{eq:1.6}
\f_t(y_0,t_0) + \min_{a\in \R^N}\left\{ \sqrt{E(y_0)}a\cdot \nabla \f(y_0,t_0) + \frac{1}{2}|a|^2\right\} \geq  0.
\end{equation}

Take $(y_n,t_n)\to (y_0,t_0)$ such that $|v(y_n,t_n)-v^*(y_0,t_0)|<\frac{1}{n^2}$ and $|\f(y_n,t_n)-\f(y_0,t_0)|<\frac{1}{n^2}$.
Let $a_0\in\R^N$ be arbitrary. We consider the constant control $\alpha(t) = a_0$, $t_0\leq t \leq T$ and some solutions $x_n$ of the equations
\[
\begin{cases}
\dot{x}_n(t) = \sqrt{E(x_n(t))}a_0,&\text{for }t_0<t<T,\\
x_n(t_n) = y_n.
\end{cases}
\]
We take $n$ large enough so that $0<\frac{1}{n}<\delta$ and such that 
\begin{equation}\label{eq:nestxn}
|x_n(t)-y_0|\leq C_1|a_0|(t-t_n)+|y_n-y_0|<\delta\quad\mbox{for}\,\,t_0\leq t\leq t+\frac{1}{n}. 
\end{equation}
By Lemma \ref{lem:1.1}, since $x_n$ is a solution of ($P(\alpha,y_n,t_n)$),
\[
v(y_n,t_n) \leq \frac{1}{2}\int_{t_n}^{t_n+\frac{1}{n}}|a_0|^2\d t + v\left(x_n\left(t_n+\frac{1}{n}\right),t_n+\frac{1}{n}\right).
\]
Thus, it follows from \eqref{ass:1.1} that
\[
\begin{split}
-\frac{1}{2}|a_0|^2&\leq n\left[v\left(x_n\left(t_n+\frac{1}{n}\right),t_n+\frac{1}{n}\right) - v(y_n,t_n)\right]
\\
&\leq n\left[v^*\left(x_n\left(t_n+\frac{1}{n}\right),t_n+\frac{1}{n}\right) - v^*(y_0,t_0)\right]+\frac{1}{n}
\\
&
\leq n\left[\f\left(x_n\left(t_n+\frac{1}{n}\right),t_n+\frac{1}{n}\right) -\f(y_0,t_0)\right]+\frac{1}{n}
\\
&
\leq n\left[\f\left(x_n\left(t_n+\frac{1}{n}\right),t_n+\frac{1}{n}\right) -\f(y_n,t_n)\right]+\frac{2}{n}.
\end{split}
\]
Therefore, we have
\[
\begin{aligned}
&-\frac{1}{2}|a_0|^2\leq n\left[\f\left(x_n\left(t_n+\frac{1}{n}\right),t_n+\frac{1}{n}\right) -\f(y_n,t_n)\right]+\frac{2}{n}\\
&\quad=n\int_{t_n}^{t_n+\frac{1}{n}} \frac{d}{dt}\f(x_n(t),t)\d t = n\int_{t_n}^{t_n+\frac{1}{n}} \left[\f_t(x_n(t),t) + \nabla\f(x_n(t),t)\cdot \dot{x}_n(t)\right]\d t
+\frac{2}{n}\\
&\quad=n\int_{t_n}^{t_n+\frac{1}{n}}\left[ \f_t(x_n(t),t) + \sqrt{E(x_n(t))}a_0\cdot\nabla\f(x_n(t),t)\right]\d t+\frac{2}{n}.
\end{aligned}
\]
Letting $n\to\infty$ above and using \eqref{eq:nestxn}, we thus obtain
\[
\f_t(y_0,t_0) +\sqrt{E(y_0)}a_0\cdot \nabla \f(y_0,t_0) + \frac{1}{2}|a_0|^2 \geq  0
\]
which gives \eqref{eq:1.6} since $a_0$ was arbitrary.

We now assume that $v-\f$ has a local minimum at some point $(y_0,t_0)\in \R^N\times(0,T)$, where $\f\in C^1(\R^N\times(0,T))$.
We will show that
\begin{equation}\label{eq:1.33}
\f_t(y_0,t_0) + \min_{a\in \R^N}\left\{ \sqrt{E(y_0)}a\cdot \nabla \f (y_0,t_0) + \frac{1}{2}|a|^2 \right\}\leq 0.
\end{equation}
Suppose this is not the case.
Then, there is $\theta >0$ such that
\begin{equation}\label{eq:1.12}
\f_t(y_0,t_0) + \min_{a\in \R^N}\left\{ \sqrt{E(y_0)}a\cdot \nabla \f (y_0,t_0) + \frac{1}{2}|a|^2 \right\} > \theta.
\end{equation}
Recall that
\[
\min_{a\in \R^N}\left\{ \sqrt{E(y_0)}a\cdot \nabla \f (y_0,t_0) + \frac{1}{2}|a|^2 \right\} =-\frac{1}{2} \left |\sqrt{E(y_0)}\nabla \f(y_0,t_0)\right|^2,
\]
see \eqref{eq:1.11}.
The above equality shows that left-hand side of \eqref{eq:1.12} is continuous, therefore there is $\delta_1 > 0$ such that
\begin{equation}\label{eq:1.13}
\begin{aligned}
|y-y_0| + |t-t_0| < \delta_1 &\implies \f_t(y,t) + \min_{a\in \R^N}\left\{ \sqrt{E(y)}a\cdot \nabla \f (y,t) + \frac{1}{2}|a|^2 \right\} > \theta\\
&\iff \forall a \in \R^N \,\,\,\f_t(y,t) + \sqrt{E(y)}a\cdot \nabla \f (y,t) + \frac{1}{2}|a|^2  > \theta.
\end{aligned}
\end{equation}
By our assumption, there is $0<\delta_2 <\delta_1$ such that
\begin{equation}\label{eq:1.14}
|y-y_0| + |t-t_0| < \delta_2 \implies v( y,t) -v(y_0,t_0) \geq \f(y,t) - \f(y_0,t_0).
\end{equation}
There is $0<h<\delta_2/2$ such that for any control $\alpha\in B_{L^2((t_0,T),\R^N)}(2\sqrt{\|g\|_{L^\infty}})$ and any solution $x$ of
\begin{equation}\label{eq:1.15}
\begin{cases}
\dot{x}(t) = \sqrt{E(x(t))}\alpha(t),\quad t_0<t<T,\\
x(t_0)=y_0,
\end{cases}
\end{equation}
we have
\begin{equation}\label{eq:1.42}
t\in [t_0,t_0+h]\implies |x(t) - y_0| < \delta_2/2.
\end{equation}
Indeed, take some $\alpha\in B_{L^2((t_0,T),\R^N)}(2\sqrt{\|g\|_{L^\infty}})$ and some solution $x$ of \eqref{eq:1.15}.
Then, for $t\in[t_0,T]$,
\[
|x(t) - y_0| \leq \int_{t_0}^t|\sqrt{E(x(s))}\alpha(s)|\d s\leq C_1\int_{t_0}^t|\alpha(s)|\d s \leq 2C_1\sqrt{\|g\|_{L^\infty}}\sqrt{t-t_0},
\]
where $C_1$ comes from \eqref{eq:1.46}.
Choosing $h =\le  \dfrac{\delta_2}{4C_1\sqrt{\|g\|_{L^\infty}}}\pr^2$ yields \eqref{eq:1.42}.
By the above and \eqref{eq:1.14}, we obtain
\begin{equation}\label{bwaz}
\begin{aligned}
&v(x(t_0+h),t_0+h)- v(y_0,t_0)  \geq\f(x(t_0+h),t_0 +h) - \f(y_0,t_0) \\
&\quad=\int_{t_0}^{t_0+h} \frac{d}{dt}\f(x(t),t)\d t = \int_{t_0}^{t_0+h} \left[\f_t(x(t),t) + \nabla\f(x(t),t)\cdot \dot{x}(t)\right]\d t \\
&\quad=\int_{t_0}^{t_0+h} \left[\f_t(x(t),t) + \sqrt{E(x(t))}\alpha(t)\cdot\nabla\f(x(t),t)\right]\d t
\end{aligned}
\end{equation}
for any control $\alpha\in B_{L^2((t_0,T),\R^N)}(2\sqrt{\|g\|_{L^\infty}})$ and any solution $x$ of \eqref{eq:1.15}.
By Corollary \ref{cor:1.1}, there is $\alpha\in B_{L^2((t_0,T),\R^N)}(2\sqrt{\|g\|_{L^\infty}})$ and a solution $x$ of \eqref{eq:1.15} such that
\begin{equation}\label{bwaz1}
v(y_0,t_0) = \frac{1}{2}\int_{t_0}^{t_0+h}|\alpha (t)|^2 \d t + v(x(t_0+h),t_0 +h).
\end{equation}
Combining \eqref{bwaz} and \eqref{bwaz1} yields
\[
\int_{t_0}^{t_0+h} \left[\f_t(x(t),t) + \sqrt{E(x(t))}\alpha(t)\cdot\nabla\f(x(t),t) + \frac{1}{2}|\alpha (t)|^2 \right]\d t \leq 0.
\]
On the other hand, by \eqref{eq:1.13}, we get
\[
\int_{t_0}^{t_0+h} \left[\f_t(x(t),t) + \sqrt{E(x(t))}\alpha(t)\cdot\nabla\f(x(t),t) + \frac{1}{2}|\alpha (t)|^2\right] \d t > \theta h
\]
which is a contradiction.
Thus \eqref{eq:1.33} holds.
\end{proof}

\section{Short time uniqueness}\label{sec:uniq}
For the sake of generality, we forget about the particular form of the matrix $E(x)$  and consider the Hamilton--Jacobi initial value problem of the form
\begin{equation}\label{eq.uniq:1}
\begin{cases}
u_t(x,t) + \frac{1}{2} M(x) \nabla u(x,t) \cdot\nabla u(x,t)  = 0, &x\in \R^N,\; t\in(0,T),\\
u(x,0) = g(x),&x\in \R^N.
\end{cases}
\end{equation}
We make the following assumptions:
\begin{itemize}
\item for every $x\in \R^N$, $M(x)$ is an $N\times N$ real matrix and there exists $L>0$ such that
\begin{equation}\label{eq.uniq:2}
\|M(x) -M(y)\| \leq L|x-y|,\qquad \text{for all }x,\;y\in \R^N,
\end{equation}
\item $g:\R^N\to R$ is bounded and Lipschitz continuous, i.e., there is $C>0$ such that
\begin{equation}\label{eq.uniq:16}
|g(x) - g(y)| \leq C|x-y|,\qquad \text{for all }x,\; y\in \R^N,
\end{equation}
\end{itemize}

We now consider the initial value problem. The terminal value problem can be addressed analogously.
To avoid ambiguities  we formulate the definition of the viscosity solution of \eqref{eq.uniq:1}.
\begin{definition}\label{def:2}
We say that a bounded and upper semi-continuous function $v:\R^N\times[0,T)\to \R$ is a \emph{viscosity subsolution} of the initial value problem \eqref{eq.uniq:1} if $v(x,0) \leq g(x)$ for $x\in \R^N$
and for every $\f \in C^1(\R^N\times (0,T))$
\begin{equation}\label{eq:1.37a}
\begin{cases}
\text{ if $v-\f$ has a local maximum at $(y_0,t_0)\in \R^N\times (0,T)$, then}\\
\quad\f_t (y_0,t_0) + \frac{1}{2}M(y_0) \nabla \f(y_0,t_0)\cdot \nabla \f (y_0,t_0) \leq  0.
\end{cases}
\end{equation}
We say that a bounded and lower semi-continuous function $v:\R^N\times[0,T)\to \R$ is a \emph{viscosity supersolution} of the terminal value problem \eqref{eq.uniq:1} if $v(x,0) \geq g(x)$ for $x\in \R^N$
and for every $\f \in C^1(\R^N\times (0,T))$
\begin{equation}\label{eq:1.38a}
\begin{cases}
\text{ if $v-\f$ has a local minimum at $(y_0,t_0)\in \R^N\times (0,T)$, then}\\
\quad\f_t (y_0,t_0) + \frac{1}{2}M(y_0) \nabla \f(y_0,t_0)\cdot \nabla \f (y_0,t_0) \geq  0.
\end{cases}
\end{equation}
A function $v$ is a \emph{viscosity solution} of \eqref{eq:1.2} if it is a viscosity subsolution and a viscosity supersolution of \eqref{eq:1.2}.
\end{definition}

Let us recall the notion of sub- and superdifferential, see, e.g., \cite{BCD}.
Let $u:\R^N\times (0,T)\to \R$  and $(x,t)\in \R^N \times (0,T)$.
The \emph{superdifferential} $D^+ u(x,t)$ of $u$ at $(x,t)$ is defined as follows: for $(b,\xi)\in \R\times \R^N$,
\begin{gather*}
(b,\xi) \in D^+ u (x,t)\iff\\
 \limsup_{(y,s)\to (x,s)} \frac{u(y,s)-u(x,t) - b(s-t) - \xi \cdot (y-x)}{|y-x| +|s-t|}\leq 0.
\end{gather*}
The \emph{subdifferential} $D^- u (x,t)$ of $u$ at $(x,t)$ is given by
\[
D^- u(x,t) := - D^+ (-u)(x,t).
\]

\begin{remark}\label{rem:3}
We recall that if $u:\R^N\times (0,T)\to \R$ and $(x,t)\in \R^N\times (0,T)$
then $(b,\xi)\in D^+u(x,t)$ if and only if there exists $\f \in C^1(\R^N\times (0,T))$ such that $(b,\xi) = (\f_t (x,t), \nabla\f (x,t))$ and $u-\f$ has local maximum at $(x,t)$.
For the proof see e.g. Lemma 1.7, \S II in \cite{BCD} (it is assumed that $u$ is continuous there, but the inspection of the proof reveals that $u$ may be arbitrary function).
The similar result holds for the subdifferential $D^- u(x,t)$. In \cite{CIL} the second order sub- and superjets and their closures were introduced to deal with the second order PDEs. Thus, following \cite{CIL}, we introduce the closure of sub- and superdifferential (this concept is known in the field of non-smooth analysis under different names like ''limiting subdifferential'' or ''general subgradient'' in \cite{RW}).
The closure of the superdifferential $\overline{D}^+ u(x,t)$ of $u$ at $(x,t)$ is given by
\begin{gather*}
\overline{D}^+u(x,t) := \Big\{ (b,\xi)\in \R\times \R^N\mid \exists (x_n,t_n,b_n,\xi_n)\in \R^N\times (0,T)\times \R \times \R^N\\
\text{ such that }(b_n,\xi_n)\in D^+ u(x_n,t_n)\\
\text{ and } (x_n,t_n,u(x_n,t_n),b_n,\xi_n)\to (x,t,u(x,t),b,\xi)\Big\}.
\end{gather*}
The closure of the subdifferential $\overline{D}^- u (x,t)$ of $u$ at $(x,t)$ is defined in the similar manner.
Then, since $M$ is continuous, Definition \ref{def:2} may be equivalently stated as follows.

A bounded and upper semi-continuous function $v:\R^N\times[0,T)\to \R$ is a \emph{viscosity subsolution} of the initial value problem \eqref{eq.uniq:1} if $v(x,0) \leq g(x)$ for $x\in \R^N$
and for every $(x,t) \in R^N\times (0,T)$
\[
\text{ if $(b,\xi)\in \overline{D}^+ u(x,t)$, then } b +  \frac{1}{2}M(x)\xi \cdot \xi \leq  0.
\]
A bounded and lower semi-continuous function $v:\R^N\times[0,T)\to \R$ is a \emph{viscosity supersolution} of the initial value problem \eqref{eq.uniq:1} if $v(x,0) \geq g(x)$ for $x\in \R^N$
and for every $(x,t) \in R^N\times (0,T)$
\[
\text{ if $(b,\xi)\in \overline{D}^- u(x,t)$, then } b +  \frac{1}{2}M(x)\xi \cdot \xi \geq  0.
\]
\end{remark}

\begin{theorem}\label{thm.uniq:1}
Let $M:\R^N\to \R^{N\times N}$ be a matrix-valued map satisfying \eqref{eq.uniq:2} and let $g:\R^N\to \R$ satisfy \eqref{eq.uniq:16}.
Let $L,C$ be from \eqref{eq.uniq:2} and \eqref{eq.uniq:16}. We define
\[
T:= \frac{2}{LC}\;.
\]
If $u$ is a viscosity subsolution of \eqref{eq.uniq:1} and $v$ is a viscosity supersolution of \eqref{eq.uniq:1}, then $u\leq v$ in $\R^N\times [0,T)$.
If $u$ is a viscosity solution of \eqref{eq.uniq:1}, then for every $\tau\in(0,T), R>0$ there are $L_\tau >0, L_{\tau,R}>0$ such that
\begin{equation}\label{Ltau}
|u(x,t)- u(y,t)| \leq L_\tau |x-y|,\quad \text{for all }x,\,y \in \R^N,\;t\in [0,\tau],
\end{equation}
\begin{equation}\label{LtauR}
|u(x,t)- u(x,s)| \leq L_{\tau,R}|t-s|,\quad \text{for all }|x|<R,\;t,s \in [0,\tau].
\end{equation}
Moreover, if in addition $M$ is bounded, then for every $\tau\in(0,T)$, there is $L_\tau >0$ such that
\begin{equation}\label{Ltauxt}
|u(x,t)- u(y,s)| \leq L_\tau( |x-y|+|t-s|),\quad \text{for all }x,\,y \in \R^N,\;t\,,s\in [0,\tau].
\end{equation}
\end{theorem}
\begin{proof}
We divide the proof into two parts.

\noindent \textbf{Part I: Comparison and Lipschitz continuity in the space variable.}

We will prove that for every $\gamma>0$ and all $(x,y,t)\in \R^{2N}\times (0,T)$
\begin{equation}\label{eq:gammaineq}
u(x,t) - v(y,t) \leq \frac{CT}{T-t}\le \gamma + |x - y|^2 \pr ^{1/2}.
\end{equation}
This will imply $u\leq v$ and in particular, if $u$ is a viscosity solution of \eqref{eq.uniq:1}, then
\[
|u(x,t)-u(y,t)| \leq \frac{CT}{T-t}|x-y|,\qquad\text{for all }t\in [0,T),\,x,\,y\in \R^N,
\]
which gives \eqref{Ltau}.

Suppose \eqref{eq:gammaineq} is not true. Then there are $\gamma_0,\sigma>0$ such that for $0<\gamma<\gamma_0$
\[
\sup_{(t,x,y)\in [0,T)\times\R^{2N}} \left\{u(x,t) - v(y,t) - \frac{CT}{T-t}\le \gamma + |x - y|^2 \pr ^{1/2}\right\}>\sigma.
\]
Let $\mu>0$ be such that
\begin{equation}\label{eq:psi}
\sup_{(t,x,y)\in [0,T)\times\R^{2N}} \left\{u(x,t) - v(y,t) -\frac{\mu}{T-t}- \frac{CT}{T-t}\le \gamma + |x - y|^2 \pr ^{1/2}\right\}>\sigma.
\end{equation}
For $\delta>0$, we consider the function $\Psi:\R^{2N}\times [0,T)\to \R$ given by
\[
\Psi(x,y,t) :=  u(x,t) - v(y,t) - \frac{\mu}{T-t} - \frac{CT}{T-t}\le \gamma + |x-y|^2\pr ^{1/2} - \delta e^{Kt}\le \langle x \rangle  + \langle y \rangle\pr,
\]
where $\langle x\rangle := (1 +2(\|M(0)\|^2+ L^2|x|^2))^{1/2}$, $x$, $y\in \R^N$, $t\in [0,T)$ and
\[
K:=\frac{LCT}{\sqrt{2}\mu}(\|u\|_\infty+\|v\|_\infty)+1.
\]
We notice that by \eqref{eq.uniq:2}, $\|M(x)\|\leq \langle x\rangle$.
It follows from \eqref{eq:psi} that there is 
$\delta_0>0$ such that for every $0<\delta<\delta_0$
\[
\sup_{(t,x,y)\in [0,T)\times\R^{2N}} \Psi(x,y,t)>0.
\]
Since $u$ and $v$ are bounded, 
\[
\begin{aligned}
\Psi(x,y,t)  \to -\infty,
\end{aligned}
\]
if $|x|\to \infty$ or $|y|\to \infty$ or $t\to T^-$.
Therefore $\sup_{\R^{2N}\times[0,T)}\Psi = \sup_{K\times [0,a]}\Psi$ for some compact $K\subset \R^{2N}$ and $a>0$.
Since $\Psi$ is upper semi-continuous, there exists $(\overline{x},\overline{y},\overline{t})\in \R^ {2N}\times [0,T)$ such that
\[
\Psi(\overline{x}, \overline{y}, \overline{t})  = \max_{\R^{2N}\times[0,T)}\Psi>0.
\]
We notice that $\overline{t} \not= 0$ as if $\overline{t} = 0$
then by the definition of a viscosity sub- and supersolution and by \eqref{eq.uniq:16},
\[
\begin{aligned}
0 &< \Psi(\overline{x},\overline{y},0) \leq g(\overline{x}) - g(\overline{y}) -\frac{\mu}{T} - C\le \gamma+|\overline{x} - \overline{y}|^2\pr ^{1/2} - \delta e^{K\overline t}\le \langle \overline{x}\rangle + \langle\overline{y}\rangle\pr \\
&\leq C|\overline{x} - \overline{y}|- C\le \gamma + |\overline{x}-\overline{y}|^2\pr ^{1/2} < 0
\end{aligned}
\]
which is a contradiction. We also note that \eqref{eq:psi} implies
\begin{equation}\label{eq:T}
\frac{1}{T-\overline t}< \frac{\|u\|_\infty+\|v\|_\infty}{\mu}.
\end{equation}

We now use Theorem 8.3 of \cite{CIL}(\footnote{
In Theorem 8.3 of \cite{CIL} the closures of the parabolic subjects $\overline{\mathcal{P}}^{1,+}u(x,t)$ and superjets 
$\overline{\mathcal{P}}^{1,-}u(x,t)$ are used. However it is easy to see that the projections onto the ``gradient'' components of
$\overline{\mathcal{P}}^{1,+}u(x,t)$ and 
$\overline{\mathcal{P}}^{1,-}u(x,t)$ are subsets of $\overline{D}^+ u(x,t)$ and $\overline{D}^- u(x,t)$ respectively. Of course the full force of
Theorem 8.3 of \cite{CIL} is not needed and the conclusions \eqref{eq.uniq:11a}-\eqref{eq.uniq:11} can be obtained by a standard doubling
of the time variable and penalization argument, however we cite Theorem 8.3 of \cite{CIL} here to shorten the proof since it is a convenient reference. We encourage the readers to use the above suggestion to show \eqref{eq.uniq:11a}-\eqref{eq.uniq:11} without referring to Theorem 8.3 of \cite{CIL}.
}) 
to obtain that there exist $b_1$, $b_2\in \R$ such that
\begin{gather}
\le b_1, \frac{CT}{T-\overline{t}}\frac{\overline{x}-\overline{y}}{\le \gamma + |\overline{x}-\overline{y}|^2\pr ^{1/2}} + 2\delta L^2 e^{K\overline t}\frac{\overline{x}}{\langle \overline{x}\rangle}\pr \in \overline{D}^+u(\overline{t},\overline{x}),\label{eq.uniq:11a}\\
\le - b_2, \frac{CT}{T-\overline{t}}\frac{\overline{x}-\overline{y}}{\le \gamma + |\overline{x}-\overline{y}|^2\pr ^{1/2}} - 2\delta L^2e^{K\overline t}\frac{\overline{y}}{\langle \overline{y}\rangle}\pr \in \overline{D}^-v(\overline{t},\overline{y}),\label{eq.uniq:11b}\\
b_1 + b_2 =  \frac{\mu}{(T-\overline{t})^2} + \frac{CT}{(T-\overline{t})^2}\le \gamma + |\overline{x}-\overline{y}|^2\pr ^{1/2}
+\delta Ke^{K\overline t}\le \langle \overline{x}\rangle + \langle\overline{y}\rangle\pr.\label{eq.uniq:11}
\end{gather}
By Remark \ref{rem:3} we now have
\[
\begin{aligned}
b_1 + \frac{1}{2}M(\overline{x})\le \frac{CT}{T-\overline{t}}\frac{\overline{x}-\overline{y}}{\le \gamma + |\overline{x}-\overline{y}|^2\pr ^{1/2}} + 2\delta L^2e^{K\overline t}\frac{\overline{x}}{\langle \overline{x}\rangle}\pr \\
\cdot\le \frac{CT}{T-\overline{t}}\frac{\overline{x}-\overline{y}}{\le \gamma + |\overline{x}-\overline{y}|^2\pr ^{1/2}} + 2\delta L^2e^{K\overline t}\frac{\overline{x}}{\langle \overline{x}\rangle} \pr \leq 0
\end{aligned}
\]
and consequently,
\begin{equation}\label{eq.uniq:8}
\begin{split}
b_1 &+ \frac{1}{2} \le \frac{CT}{T-\overline{t}}\pr^2 M(\overline{x})\le \frac{\overline{x}-\overline{y}}{\le \gamma + |\overline{x}-\overline{y}|^2\pr ^{1/2}} \pr \cdot \le \frac{\overline{x}-\overline{y}}{\le \gamma + |\overline{x}-\overline{y}|^2\pr ^{1/2}} \pr\\
&\leq - \delta L^2 e^{K\overline t}M(\overline{x})\le \frac{CT}{T-\overline{t}}\frac{\overline{x}-\overline{y}}{\le \gamma + |\overline{x}-\overline{y}|^2\pr ^{1/2}} \pr \cdot \frac{\overline{x}}{\langle \overline{x}\rangle}\\
&-\delta L^2e^{K\overline t}M(\overline{x})\frac{\overline{x}}{\langle \overline{x}\rangle}\cdot \le \frac{CT}{T-\overline{t}}\frac{\overline{x}-\overline{y}}{\le \gamma + |\overline{x}-\overline{y}|^2\pr ^{1/2}} \pr
-2\delta^2 L^4e^{2K\overline t}M(\overline{x})\frac{\overline{x}}{\langle \overline{x}\rangle}\cdot \frac{\overline{x}}{\langle \overline{x}\rangle}.
\end{split}
\end{equation}
In view of the inequalities
\[
\frac{|\overline{x}-\overline{y}|}{\le \gamma + |\overline{x}-\overline{y}|^2\pr ^{1/2}}\leq 1,\quad \frac{\sqrt{2}L|\overline{x}|}{\langle \overline{x}\rangle} \leq 1,
\]
$\|M(x)\|\leq \langle x\rangle$ and \eqref{eq:T}, we obtain from \eqref{eq.uniq:8},
\begin{equation}\label{eq.uniq:9}
\begin{split}
b_1 &+ \frac{1}{2} \le \frac{CT}{T-\overline{t}}\pr^2 M(\overline{x})\le \frac{\overline{x}-\overline{y}}{\le \gamma + |\overline{x}-\overline{y}|^2\pr ^{1/2}} \pr \cdot \le \frac{\overline{x}-\overline{y}}{\le \gamma + |\overline{x}-\overline{y}|^2\pr ^{1/2}} \pr
\\
&
\leq 
\delta e^{K\overline t}\langle \overline{x}\rangle\left(\frac{LCT}{\sqrt{2}\mu}(\|u\|_\infty+\|v\|_\infty)+\delta L^2 e^{KT}\right)\leq
\delta K e^{K\overline t}\langle \overline{x}\rangle
\end{split}
\end{equation}
if $\delta\leq 1/(L^2 e^{KT})$.

Using the definition of the viscosity supersolution and  Remark \ref{rem:3} we have
\[
\begin{aligned}
- b_2 +\frac{1}{2}M(\overline{y})\le  \frac{CT}{T-\overline{t}}\frac{\overline{x}-\overline{y}}{\le \gamma + |\overline{x}-\overline{y}|^2\pr ^{1/2}} - 2\delta L^2e^{K\overline t}\frac{\overline{y}}{\langle \overline{y}\rangle}\pr \\
\cdot \le  \frac{CT}{T-\overline{t}}\frac{\overline{x}-\overline{y}}{\le \gamma + |\overline{x}-\overline{y}|^2\pr ^{1/2}} - 2\delta L^2e^{K\overline t}\frac{\overline{y}}{\langle \overline{y}\rangle}\pr \geq 0.
\end{aligned}
\]
A similar argument as above yields
\begin{equation}\label{eq.uniq:10}
b_2 -\frac{1}{2}\le \frac{CT}{T-\overline{t}} \pr^2 M(\overline{y}) \le \frac{\overline{x}-\overline{y}}{\le \gamma + |\overline{x}-\overline{y}|^2\pr ^{1/2}}\pr \cdot \frac{\overline{x}-\overline{y}}{\le \gamma + |\overline{x}-\overline{y}|^2\pr ^{1/2}}\leq \delta K e^{K\overline t}\langle \overline{y}\rangle
\end{equation}
if $\delta\leq 1/(L^2 e^{KT})$.
We now add \eqref{eq.uniq:9} and \eqref{eq.uniq:10} and use \eqref{eq.uniq:11} to obtain
\begin{multline}\label{eq.uniq:13}
\frac{\mu}{(T-\overline{t})^2} + \frac{CT}{(T-\overline{t})^2}\le \gamma + |\overline{x}-\overline{y}|^2\pr ^{1/2} \\
+ \frac{1}{2}\le \frac{CT}{T-\overline{t}} \pr^2 \le M(\overline{x}) - M(\overline{y}) \pr \le \frac{\overline{x}-\overline{y}}{\le \gamma + |\overline{x}-\overline{y}|^2\pr ^{1/2}}\pr \cdot \frac{\overline{x}-\overline{y}}{\le \gamma + |\overline{x}-\overline{y}|^2\pr ^{1/2}}\leq 0.
\end{multline}
By \eqref{eq.uniq:2}
\[
\le M(\overline{x}) - M(\overline{y}) \pr \le \frac{\overline{x}-\overline{y}}{\le \gamma + |\overline{x}-\overline{y}|^2\pr ^{1/2}}\pr \cdot \frac{\overline{x}-\overline{y}}{\le \gamma + |\overline{x}-\overline{y}|^2\pr ^{1/2}}\geq -L|\overline{x} - \overline{y}|,
\]
so, after  including $CTL =2$, \eqref{eq.uniq:13} implies 
\[
\frac{\mu}{(T-\overline{t})^2} + \frac{CT}{(T-\overline{t})^2}\le \gamma + |\overline{x}-\overline{y}|^2\pr ^{1/2} - \frac{CT}{(T-\overline{t})^2}|\overline{x}-\overline{y}|\leq 0.
\]
Hence
\[
\frac{\mu}{T^2}\leq 0,
\]
which is contradiction.
Therefore \eqref{eq:gammaineq} must hold.

\noindent\textbf{Part II: Lipschitz continuity in the time variable.}

The proof of the Lipschitz continuity with respect to the time variable is standard but we will sketch it for completeness. Let $0<\tau<T, R>0$. We set
\[
L_{\tau,R}=\max_{|x|\leq R+2}\frac{1}{2}\|M(x)\|L_{\tau}^2,
\]
where $L_\tau$ is from \eqref{Ltau}. For $\varepsilon>0$ and $(x,t)\in \R^N\times(0,T)$ we define the sup-convolution of $u$
\[
u^\varepsilon(x,t)= \sup_{(y,s)\in \R^N\times(0,T)}\left\{u(y,s)-\frac{1}{\varepsilon}(|x-y|^2+(t-s)^2)\right\}.
\]
It is well known that the function $u^\varepsilon$ is semi-convex, Lipschitz continuous, $u^\varepsilon\to u$ pointwise as $\varepsilon\to 0$ and it is easy to see that it satisfies
estimate \eqref{Ltau}. Moreover, see e.g. \cite{JLS} or Lemma A.3 of \cite{CCKS}, if $a_\varepsilon:=(3\varepsilon\|u\|_\infty)^{\frac{1}{2}}<1$ then $u^\varepsilon$
satisfies a.e. in $B(0,R+1)\times(a_\varepsilon,\tau-a_\varepsilon)$
\begin{equation}\label{eq:ueps}
(u^\varepsilon)_t(x,t)\leq L_{\tau,R}.
\end{equation}
If we now take $\eta_\delta, \delta>0$ to be standard mollifiers in $\R^{N+1}$ (see e.g. \cite{Evans}) and define $u_{\varepsilon}^{\delta}:=u^{\varepsilon}*\eta_\delta$ then, taking convolution of both sides of \eqref{eq:ueps} with $\eta_\delta$ and using $(u^{\varepsilon}_{\delta})_t=(u^{\varepsilon})_t*\eta_\delta$,
we get for sufficiently small $\delta$ that, for every $(x,t)\in B(0,R)\times(2a_\varepsilon,\tau-2a_\varepsilon)$,
\[
(u^{\varepsilon}_{\delta})_t(x,t)\leq L_{\tau,R}.
\]
The above implies that for every $x\in B(0,R), 2a_\varepsilon<t\leq s<\tau-2a_\varepsilon$,
\[
u^{\varepsilon}_{\delta}(x,s)-u^{\varepsilon}_{\delta}(x,t)\leq L_{\tau,R}(s-t)
\]
and, by letting $\delta\to 0$ and then $\varepsilon\to 0$, we obtain that for every $x\in B(0,R), 0<t\leq s<\tau$
\[
u(s,x)-u(t,x)\leq L_{\tau,R}(s-t).
\]
The inequality 
\[
u(s,x)-u(t,x)\geq -L_{\tau,R}(s-t)
\]
is obtained similarly by considering the the inf-convolutions of $u$
\[
u_\varepsilon(x,t)= \inf_{(y,s)\in \R^N\times(0,T)}\left\{u(y,s)+\frac{1}{\varepsilon}(|x-y|^2+(t-s)^2)\right\}.
\]
This proves \eqref{LtauR}.

If $M$ is bounded, the estimate \eqref{Ltauxt} follows since the constants $L_{\tau,R}$ do not depend on $R$.
\end{proof}

We finish with an example which illustrates two phenomena. On the one hand it shows that the time of preservation of Lipschitz continuity in Theorem \ref{thm.uniq:1} is optimal. On the other hand, it shows that one cannot expect higher regularity of solutions to the degenerate Hamilton--Jacobi equation of the form \eqref{eq:1.1} than ${1/2}$-H\"{o}lder continuity. This is an essential information. As mentioned in the Introduction, one would obtain uniqueness of viscosity solutions to \eqref{eq:1.1} for any time $T>0$, if such solutions were $C^{1/2+\epsilon}$ regular in the spacial variable. A rather standard modification of the doubling variables method gives such a claim. The example below thus shows that we face a much more subtle problem when dealing with uniqueness.

We consider a one-dimensional version of \eqref{eq.uniq:1} 
\begin{equation}\label{eq:uniq.1.1}
\begin{cases}
u_t(x,t) + \frac{1}{2} |x| |u_x(x,t)|^2  = 0, &x\in \R,\; t\in(0,T),\\
u(x,0) = g(x), &x\in \R,
\end{cases}
\end{equation}
with a particular choice of
\begin{equation}\label{eq:g1}
g(x)=\begin{cases}
-1, &\text{if }x\leq -1,\\
x, &\text{if }-1<x<1,\\
1, &\text{if }x\geq 1.
\end{cases}
\end{equation}
Notice that $g$ as well as $M(x)=|x|$ satisfy \eqref{eq.uniq:16} with $C=1$ and \eqref{eq.uniq:2} with $L=1$.

\begin{proposition}\label{ostatnia}
There exists a viscosity solution $u$ to \eqref{eq:uniq.1.1}-\eqref{eq:g1} such that $u(\cdot, t)$ is Lipschitz continuous for $0\leq t<2$ and $u(\cdot, t)$ is only $1/2$-H\"{o}lder continuous for $t>2$.
\end{proposition}
\begin{proof}
We use the results of \cite{CW} where a viscosity solution to \eqref{eq:uniq.1.1} is constructed. The viscosity solution constructed in \cite{CW} has the 
form
\begin{equation}\label{cw}
u(x,t)=v(A(x),t),
\end{equation}
where $A(x)=2\mathrm{sign}(x)\sqrt{|x|}$ and $v$ solves $v_t+\frac{1}{2}|v_x(x,t)|^2  = 0$. For $v$ we have a representation
given by the Hopf--Lax formula,
\begin{equation}\label{HL}
v(x,t)=\min_{y\in\R}\left(v_0(y)+\frac{|x-y|^2}{2t}\right).
\end{equation}
Observe that $v_0(y)=g(A^{-1}(y))$, where $A^{-1}(y) = \frac{1}{4}\mathrm{sign}(y)y^2$, so
\begin{equation}\label{eq:g2}
v_0(y)=\begin{cases}
-1,&\text{if } y\leq -2,\\
\frac{1}{4}\mathrm{sign}(y)y^2, &\text{if }-2<y<2,\\
1, &\text{if }y\geq 2.
\end{cases}
\end{equation}
Plugging $v_0(y)$ in \eqref{HL}, we are in a position to find $v$. 

First, we restrict our attention to the region $x\in \R$, $0\leq t < 2$.
After considering a number of cases, we arrive at 
\begin{equation}\label{eq:last}
v(x,t)  = 
\begin{cases}
-1, & \text{if }x \leq -2, \, 0\leq t < 2,\\
\dfrac{(x+2)^2}{2t}-1, &\text{if }-2<x\leq 0,\, 0\leq t< 2,\, t\geq x+2,\\[6pt]
\dfrac{x^2}{2(t-2)},&\text{if }-2<x \leq 0,\, 0\leq t< 2,\, t< x+2,\\[6pt]
\dfrac{x^2}{2(t+2)},&\text{if } x >0, \,0\leq t < 2,\, t \geq \frac{1}{2}x^2 -2,\\
1,&\text{if }, x >0, \,0\leq t < 2,\, t < \frac{1}{2}x^2 -2.
\end{cases}
\end{equation}
The following picture illustrates it.
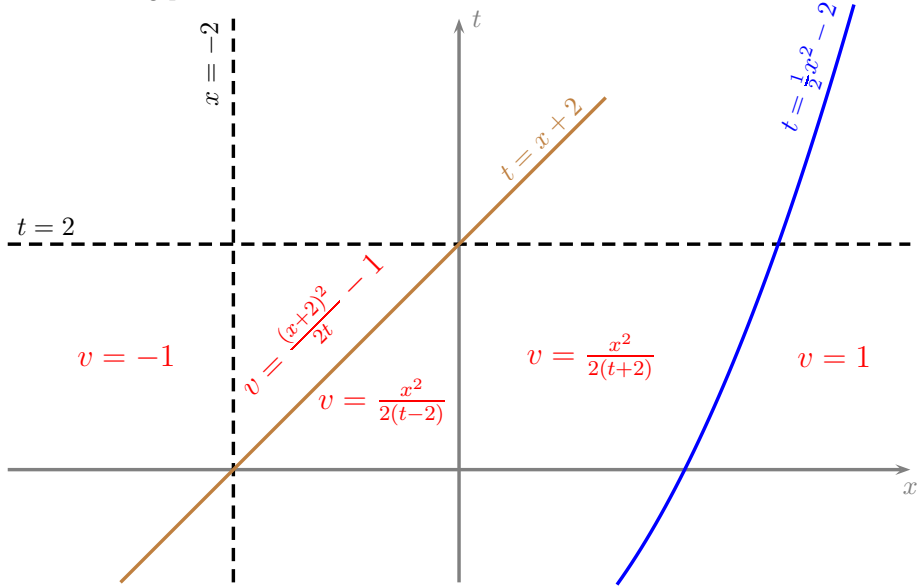
\begin{center}
\psset{algebraic,unit=1.5cm}
\begin{pspicture}[showgrid=salse](-4,-1)(4,4)
\psline[linecolor = gray, linewidth = 1.3pt]{->}(-4,0)(4,0)
\uput[270](4,0){\color{gray}{$x$}}
\psline[linecolor = gray, linewidth = 1.3pt]{->}(0,-1)(0,4)
\uput[0](0,4){\color{gray}{$t$}}
\psline[linestyle = dashed,linewidth = 1.3pt](-4,2)(4,2)
\uput[45](-4,2){$t = 2$}
\psplot[linecolor = blue,linewidth = 1.3pt]{1.4}{3.5}{0.5*x*x-2}
\rput[l]{72}(2.93,3){\color{blue}{$t=\frac{1}{2}x^2-2$}}
\psline[linestyle = dashed, linewidth = 1.3pt](-2,-1)(-2,4)
\rput[u]{90}(-2.2,3.6){$x= -2$}
\psline[linecolor = brown, linewidth = 1.3pt](-3,-1)(1.3,3.3)
\rput[l]{45}(0.4,2.6){\color{brown}{$t=x+2$}}
\uput[0](-3.5,1){\large\color{red}{$v = -1$}}
\rput[l]{45}(-1.9,0.7){\large\color{red}{$v = \frac{(x+2)^2}{2t}-1$}}
\rput[l]{0}(-1.25,0.6){\large\color{red}{$v = \frac{x^2}{2(t-2)}$}}
\rput[l]{0}(0.6,1){\large\color{red}{$v = \frac{x^2}{2(t+2)}$}}
\rput[l]{0}(3,1){\large\color{red}{$v = 1$}}
\end{pspicture}
\end{center}

In view of \eqref{eq:last}, \eqref{cw} and the definition of $A$, we clearly see that $u$ is Lipschitz continuous for $0\leq t<2$ and $x\in \R$.

Next, using \eqref{HL}, we notice that for $x\in(-2,0]$ and $t>2$,
\[
v(x,t)=\frac{(x+2)^2}{2t}-1.
\] 
Hence and by \eqref{cw}, for $x\in(-1,0]$ and $t>2$,
\[
u(x,t)=\frac{(A(x)+2)^2}{2t}-1 = \frac{-2x - 4\sqrt{-x} +2}{t}-1.
\]
This means that $u$ ceases to be Lipschitz continuous
for $t>2$ and it is exactly $1/2$-H\"{o}lder continuous there.
\end{proof}

{\bf Acknowledgements.} T.C. was supported by the National Science Center of Poland grant SONATA BIS 7 number UMO-2017/26/E/ST1/00989.
This work was partially supported by the grant 346300 for IMPAN from the Simons Foundation and the matching 2015-2019 Polish MNiSW fund. T.C and J.S. are grateful to S\l{}awomir Plaskacz for helpful discussions.


\begin{thebibliography}{99}

\bibitem{BCD}
\textsc{M. Bardi, I. Capuzzo-Dolcetta,:}
\textit{Optimal control and viscosity solutions of Hamilton-Jacobi-Bellman equations},
Birkh\"{a}user Boston, Inc., Boston, MA, 1997.

\bibitem{BSS}
\textsc{R. Beals, D. Sattinger, J. Szmigielski,:}
\textit{Multipeakons and the classical moment problem},
Adv. Math. {\bf 154} (2000), 229--257.

\bibitem{Bhatia}
\textsc{R. Bhatia,:}
\textit{Matrix Analysis},
Springer-Verlag, New York, 1997.

\bibitem{CCKS}
\textsc{L. Caffarelli, M.G. Crandall, M. Kocan, A. \'{S}wi\k{e}ch,:}
 \textit{On viscosity solutions of fully nonlinear equations with measurable ingredients},
 Comm. Pure Appl. Math. {\bf 49} (1996), no. 4, 365--397.

\bibitem{CH}
\textsc{R. Camassa, D. Holm,:}
\textit{An integrable shallow water equation with peaked solitons},
Phys. Rev. Lett. {\bf 71} (1993), 1661--1664.

\bibitem{CS}
\textsc{P. Cannarsa, C. Sinestrari,:}
\textit{Semiconcave functions, Hamilton-Jacobi equations, and optimal control},
Birkh\"{a}user Boston, Inc., Boston, MA. 2004.

\bibitem{CGKM}
\textsc{T. Cie\'slak, M. Gaczkowski, M. Kubkowski, M. Ma\l{}ogrosz,:}
\textit{Multipeakons viewed as geodesics},
Bull. Polish Acad. Sci. {\bf 65} (2017), 153--164.

\bibitem{CW}
\textsc{T. Cie\'slak, H. Wakui,:}
\textit{Existence of solutions to a one-dimensional Hamilton-Jacobi equation with a degenerate Hamiltonian}, submitted.

\bibitem{CIL}
\textsc{M.G. Crandall, H. Ishii, P.-L. Lions,:}
\textit{User's guide to viscosity solutions of second order partial differential equations},
Bull. Amer. Math. Soc. (N.S.) {\bf 27} (1992), no. 1, 1--67.

\bibitem{CDL}
\textsc{A. Cutr\`i, F. Da Lio,:}
\textit{Comparison and existence results for evolutive non-coercive first-order Hamilton-Jacobi equations},
ESAIM Control Optim. Calc. Var. {\bf 13} (2007), no. 3, 484--502.

\bibitem{DLL1}
\textsc{F. Da Lio, O. Ley,:}
\textit{Uniqueness results for second-order Bellman-Isaacs equations under quadratic growth assumptions and applications},
SIAM J. Control Optim. {\bf 45} (2006), no. 1, 74--106.

\bibitem{DLL2}
\textsc{F. Da Lio, O. Ley,:}
\textit{Convex Hamilton-Jacobi equations under superlinear growth conditions on data},
Appl. Math. Optim. {\bf 63} (2011), no. 3, 309--339.

\bibitem{Evans}
\textsc{L.C. Evans,:}
\textit{Partial differential equations},
American Mathematical Society, Providence, RI, 1998.

\bibitem{FS}
\textsc{W. Fleming, H.M.  Soner,:}
\textit{Controlled Markov processes and viscosity solutions. Second edition},
Springer, New York, 2006.

\bibitem{FPR}
\textsc{H. Frankowska, S. Plaskacz, T. Rze\.zuchowski,:}
\textit{Measure viability theorems and the Hamilton-Jacobi-Bellman equation},
J. Differential Equations {\bf 116} (1995), 265--305.

\bibitem{galbraith}
\textsc{G.N. Galbraith,:}
\textit{Extended Hamilton-Jacobi characterization of value functions in optimal control},
SIAM J. Control Optim. {\bf 39} (2000), 281--305.

\bibitem{GH}
\textsc{K. Grunert, H. Holden,:}
\textit{The general peakon-antipeakon solution for the Camassa-Holm equation},
J. Hyperbolic Differential Equations {\bf 13} (2016), 353--380.

\bibitem{JLS}
\textsc{R. Jensen, P.-L. Lions, P. Souganidis,:}
\textit{A uniqueness result for viscosity solutions of second order fully nonlinear partial differential equations},
Proc. Amer. Math. Soc. {\bf 102} (1988), no. 4, 975--978. 

\bibitem{K}
\textsc{W. Kry\'nski,:}
\textit{Dissipative prolongations of the mulitpeakon solutions to the Camassa-Holm equation},
J. Differential Equations {\bf 266} (2019), 1832--1850.

\bibitem{RW}
\textsc{R.T. Rockafellar, R.J.-B. Wets,:}
\textit{Variational analysis},
Springer-Verlag, Berlin, 1998, (3rd printing 2009).

\bibitem{Walter}
\textsc{W.  Walter,:}
\textit{Ordinary differential equations},
Springer-Verlag, New York, 1998.

\end{thebibliography}
\end{document}